\documentclass[leqno,12pt]{amsart} 
\usepackage[svgnames]{xcolor}	
\usepackage{amssymb}
\usepackage{amsmath}
\usepackage{amsthm}
\usepackage{amscd}
\usepackage{mathtools}
\usepackage{enumitem}
\usepackage[pdftex]{hyperref}
\usepackage{wrapfig}
\usepackage{tikz}
\usepackage{ulem}

\usepackage{tikz}
\usepackage{tikz-cd}
\usetikzlibrary{positioning}
\usetikzlibrary{matrix,arrows,decorations.pathmorphing}

\usetikzlibrary{intersections}
\usetikzlibrary{calc}
\usetikzlibrary{arrows.meta}

\usepackage{cleveref}

\theoremstyle{plain} 
\newtheorem{theorem}{\indent\sc Theorem}[section]
\newtheorem{lemma}[theorem]{\indent\sc Lemma}

\newtheorem{proposition}[theorem]{\indent\sc Proposition}

\theoremstyle{definition} 
\newtheorem{definition}[theorem]{\indent\sc Definition}
\newtheorem{remark}[theorem]{\indent\sc Remark}
\newtheorem{example}[theorem]{\indent\sc Example}

\newtheorem*{claim0}{\indent\sc Claim}

\numberwithin{equation}{section}

\newcommand{\R}{\mathbb{R}}

\newcommand{\N}{\mathbb{N}}

\newcommand{\cl}{\operatorname{cl}}

\newcommand{\relmiddle}[1]{\mathrel{}\middle#1\mathrel{}}

\crefname{theorem}{Theorem}{Theorems}
\crefname{lemma}{Lemma}{Lemma}
\crefname{corollary}{Corollary}{Corollaries}
\crefname{proposition}{Proposition}{Propositions}
\crefname{example}{Example}{Examples}
\crefname{section}{Section}{Sections}
\crefname{definition}{Definition}{Definition}
\begin{document}

\title[Horoboundaries of coarsely convex spaces]{Horoboundaries of coarsely convex spaces} 

\author[Ikkei Sato]{Ikkei Sato}


\renewcommand{\thefootnote}{\fnsymbol{footnote}}

\keywords{ coarse geometry, coarsely convex space, horoboundary
}





\address{Ikkei Sato \endgraf
Department of Mathematical Sciences \endgraf
Tokyo Metropolitan University \endgraf
Minami-osawa Hachioji Tokyo 192-0397 \endgraf
Japan
}
\email{sato-ikkei@tmu.ac.jp}


\begin{abstract}
A horoboundary is one of the attempts to compactify metric spaces, and is constructed using continuous functions on metric spaces. It is a concept that includes global information of metric spaces, and its correspondence with an ideal boundary constructed using geodesics has been studied in nonpositive curvature spaces such as CAT(0) spaces and geodesic Gromov hyperbolic spaces. We will introduce a certain correspondence between the horoboundary and the ideal boundary of coarsely convex spaces, which can be regarded as a generalization of spaces of nonpositive curvature.
\end{abstract}

\maketitle

\tableofcontents







\section{Introduction}
\subsection{Introduction}
Coarse geometry is the field of studying invariant properties of metric
spaces under quasi-isometry.  
The geodesic Gromov
hyperbolic spaces introduced by Gromov are invariant under
quasi-isometry and have been studied as the coarse geometric analogues
of Riemannian manifolds with negative sectional curvature.
Subsequently, the construction of analogues in the metric geometry of
Riemannian manifolds with nonpositive sectional curvature was advanced,
and CAT(0) spaces and Busemann spaces were introduced. However, these
spaces are not invariant under quasi-isometry. 
Coarsely convex spaces introduced by
Fukaya and Oguni in \cite{FO17} are coarse geometric analogues of simply
connected Riemannian manifolds of nonpositive sectional curvature. They
are a class that includes not only CAT(0) spaces and Busemann spaces,
but also geodesic hyperbolic spaces, finite-dimensional systolic
complexes, proper injective spaces.  Furthermore, in coarsely
convex spaces, as in other non-positive curvature spaces, ideal
boundaries are defined by using (quasi-)geodesic rays, the coarse
Cartan-Hadamard theorem, which states that the open cones defined on the
ideal boundary and the coarsely convex spaces themselves are coarse
homotopy equivalence, and coarsely convex spaces satisfy various
properties that allow them to be regarded as generalizations of
nonpositive curvature spaces.

The horoboundary was introduced by Gromov in \cite{Gromov+1981+183+214}, as a way to compactify proper metric spaces.
In particular, he constructed it as a concept that includes Busemann functions constructed using geodesic rays.
The elements of the horoboundary are called horofunctions, and in CAT(0) spaces all horofunctions are Busemann functions.
In particular, for the CAT(0) space, there exists a bijection between the set of Busemann functions and that of geodesic rays.
That is, the horoboundary is homeomorphic to the ideal boundary.

So what about in other non-positive curvature spaces? This is not viable.
For example, in geodesic hyperbolic spaces, by Webster and Winchester
\cite{WW03},
it was shown that there exists a surjective continuous map from the horoboundary to the ideal boundary.
However, Arosio, Fiacchi, Gontard and Guerini \cite{AFGG22}
gave an example of Gromov hyperbolic space whose horoboundary differs from
the ideal boundary.
Other similar results were shown in the Busemann spaces by Andreev
\cite{And08}.  However, in \cite{And18}, Andreev introduced the cone
metric $d_c$ for Busemann spaces $(X,d)$, 
and he showed that 
the horoboundary of $(X,d_c)$ is homeomorphic to the ideal boundary of $(X,d)$.

We note that the cone metric defined by Andreev is similar 
to the Euclidean cone metric for open cones, 
and show the following results on coarsely convex spaces.

\begin{theorem}
\label{thm:horoboundary-of-opencone} 
 Let $X$ be a proper coarsely convex
 space and $o \in X$ be a base point of $X$.  Let $\partial_o X$
 be the ideal boundary of $X$ with respect to the base point $o$ and
 let $\mathcal{O}\partial_o X$ be an open cone over $\partial_o X$.
 Then, the horoboundary $\partial_h \mathcal{O}\partial_o X$ of the
 open cone $\mathcal{O}\partial_o X$ is homeomorphic to the ideal
 boundary $\partial_o X$ of $X$.
\end{theorem}

We construct the cone metric $d_c$ on a coarsely convex space $X$. 
Due to the coarse geometric nature of coarsely convex spaces, we need to modify the formulation of
the horoboundary $\partial^c_h X$ of $X$ with the cone metric $d_c$. 
When $X$ is a Busemann space, $d_c$ coincide with the one defined by Andreev in \cite{And18},
and $\partial^c_h X$ coincide with the one studied in \cite{And18}.

We compare the horoboundary with
the ideal boundary.
Since the ideal boundary of the coarsely convex space is the set of asymptotic classes of quasi-geodesic rays, we need to take a quotient by bounded functions. The quotient is called 
the reduced horoboundary denoted by $\partial^c_h X/\sim$. For more
detail, see \cref{def:reduced-horoboundary}. We show that the ideal boundary coincides with
the reduced horoboundary.

\begin{theorem}
\label{thm-Intro:reduced-horobdry-equal-ideal-bdry} Let $(X, d)$ be a  
 coarsely convex space and let $o \in X$ be a base point of
 $X$. Let $d_c$ be the cone metric of $(X,d)$.
 Let $\partial_o X$ be the ideal boundary of $X$ with respect to
 the base point $o$ and $\partial^c_h X/\sim$ be the reduced
 horoboundary of $(X,d_c)$. Then $\partial_o X$ and $\partial^c_h X/\sim$ are
 homeomorphic.
\end{theorem}


In geodesic Gromov hyperbolic spaces, by \cite{WW03},
the ideal boundary coincides with the reduced horoboundary in the usual sense.
\cref{thm-Intro:reduced-horobdry-equal-ideal-bdry}
generalize this fact to geodesic coarsely convex spaces in a sense.

Furthermore, in geodesic $(1,0)$-coarsely convex spaces such as Busemann spaces, the reduced horoboundary with the cone metric and the horoboundary with the cone metric coincide, 
so \cref{thm-Intro:reduced-horobdry-equal-ideal-bdry} generalize a result in \cite{And18} by Andreev.

\subsection{Outline}
In \cref{section:CCS}, we introduce coarsely convex spaces and give some examples according to~\cite{FO17,fukaya2024boundaryfreeproductsmetric}.

In \cref{section:ideal-boundary}, 
we describe the construction of the ideal boundaries for coarsely convex spaces, and we discuss their properties.

In \cref{section:horoboundary}, 
we summarize some known facts on horoboundary

In \cref{section:open-cone}, we define open cones, and we study the horoboundary of open cones. We give a proof of \cref{thm:horoboundary-of-opencone}.

In \cref{section:cone-metric}, we define the cone metric for coarsely convex spaces,
 and we 
analyze the horoboundary 
of the coarsely convex space associated with the cone metric.
We complete the proof of \cref{thm-Intro:reduced-horobdry-equal-ideal-bdry}.

\section{Coarsely convex space}
\label{section:CCS}
Let $(X,d)$ be a metric space.
A \emph{bicombing} on $X$ is a map 
$\Gamma \colon X \times X \times [0, 1] \rightarrow X$ such that for 
$x, y \in X$, we have $\Gamma(x, y, 0) = x$, $\Gamma(x, y, 1) = y$.

Let $\lambda \geq 1$ and $k \geq 0$ be constants. 
A bicombing $\Gamma$ is a $(\lambda, k)$-\emph{quasi-geodesic bicombing} if
\[
 \frac{1}{\lambda}|t - s|d(x, y) - k 
 \leq d(\Gamma(x, y, t), \Gamma(x, y, s)) 
 \leq \lambda|t - s|d(x, y) + k 
\]
holds for $t, s \in [0,1]$. 
In particular, $\Gamma \colon X \times X \times [0,1]\rightarrow X$ is a \emph{geodesic-bicombing}
if $\lambda = 1$, $k = 0$ is satisfied, that is,
\[
d(\Gamma(x, y, t), \Gamma(x, y, s)) = |t - s|d(x, y) 
\]
holds
for $t, s \in [0, 1]$.

\begin{definition}
\label{def:cconvex-bicombing}
    Let $(X, d)$ be a metric space, and let 
 $\lambda \geq 1$, $k \geq 0$, $E \geq 1$, $C \geq 0$ be constants. 
 Let $\theta \colon \mathbb{R}_{\geq 0} \rightarrow \mathbb{R}_{\geq 0}$ 
 be a non-decreasing function. 
    A $(\lambda, k, E, C, \theta)$-\emph{coarsely convex bicombing} on a metric space $(X, d)$
    is a $(\lambda, k)$-quasi-geodesic bicombing 
 $\Gamma \colon X \times X \times [0, 1] \rightarrow X$ 
 with the following \cref{item:cconvex,item:theta}:
    \begin{enumerate}[label=(\roman*)]
        \item \label{item:cconvex}
              Let $x_1, x_2, y_1, y_2 \in X$ and let $a, b \in [0, 1]$. Let $y'_1 \coloneqq \Gamma(x_1, y_1, a)$ and $y'_2 \coloneqq \Gamma(x_2, y_2, b)$. Then, for $c \in [0, 1]$, we have
        \[
        d(\Gamma(x_1, y_1, ca), \Gamma(x_2, y_2, cb)) \leq (1-c)Ed(x_1, x_2) + cEd(y'_1, y'_2) + C.
        \]

        \item  \label{item:theta}
               Let $x_1, x_2, y_1, y_2 \in X$. Then for $t, s\in [0, 1]$ we have
        \[
        |td(x_1, y_1) - sd(x_2, y_2)| \leq \theta(d(x_1, x_2) + d(\Gamma(x_1, y_1, t), \Gamma(x_2, y_2, s))).
        \]
    \end{enumerate}
    
A \emph{geodesic $(E,C)$-coarsely convex bicombing}
 is a geodesic bicombing $\Gamma$ satisfying \cref{item:cconvex} in the above.
\end{definition}

\begin{remark}
 If $\Gamma$ is a geodesic bicombing, then $\Gamma$ satisfies \cref{item:theta} in 
 \cref{def:cconvex-bicombing} due to the triangle inequality.
\end{remark}

\begin{definition}
 We say that $X$ is a \emph{coarsely  convex space}
 if there are constants $\lambda \geq 1$, $k \geq 0$, $E \geq 1$, $C \geq 0$ 
 and non-decreasing function 
 $\theta \colon \mathbb{R}_{\geq 0} \rightarrow \mathbb{R}_{\geq 0}$ 
 such that
 $X$ admits a $(\lambda, k, E, C, \theta)$-coarsely convex 
 bicombing $\Gamma$. 
 In particular, 
 $X$ is a geodesic
 coarsely $(E, C)$-coarsely convex space
 if $X$ admits a 
 geodesic $(E,C)$-coarsely convex bicombing $\Gamma$.
\end{definition}

\begin{example}\label{ex:norm-is-cc}
   Let $V$ be a normed vector space.
   The bicombing of $V$ is given by affine lines.
   So $V$ is a geodesic $(1, 0)$-coarsely convex space.
\end{example}

\begin{example}
\label{ex:Busemann-is-gcc}
We say that
 a geodesic space $(X, d)$ is a \emph{Busemann space}
if any two geodesics 
$\gamma_1 \colon [0, a_1] \rightarrow X$ and
$\gamma_2 \colon [0, a_2] \rightarrow X$ 
satisfy the following inequality
\[
d(\gamma_1(ta_1), \gamma_2(ta_2))
\leq (1-t)d(\gamma_1(0), \gamma_2(0)) + td(\gamma_1(a_1), \gamma_2(a_2))
\]
for any $t \in [0, 1]$. 
It follows that a Busemann space $(X, d)$ is a unique geodesic space. 
So $X$ admits the canonical $(1,0)$-coarsely convex bicombing.
Then the Busemann space $(X, d)$ is a geodesic $(1, 0)$-coarsely convex space.
\end{example}

The following reparametrization is used to construct ideal boundaries in \cref{section:ideal-boundary} and cone metrics in \cref{section:cone-metric}.

\begin{definition}
 Let $(X, d)$ be a metric space and let 
 $\Gamma \colon X \times X \times [0, 1] \rightarrow X$ be a 
$(\lambda, k)$-quasi-geodesic bicombing on $X$. 
 A \emph{reparametrised bicombing} of 
 $\Gamma$ is a map 
 $\texttt{rp}\Gamma \colon X \times X \times \mathbb{R}_{\geq 0} \rightarrow X$ 
 defined by
    \[
    \texttt{rp}\Gamma(x, y, t) \coloneqq
    \begin{cases}
    \Gamma \left(x, y, \frac{t}{d(x, y)}\right) &\text{if} \: t \leq d(x, y)\\
    y &\text{if} \: t > d(x, y).
    \end{cases}
    \]

    In particular, for any $x, y \in X$, $\texttt{rp}\Gamma(x, y, -)|_{[0, d(x,y)]} \colon [0, d(x, y)] \rightarrow X$ is a $(\lambda, k)$-quasi-geodesic connecting $x$ and $y$. If $\Gamma$ is a geodesic bicombing, \texttt{rp}$\Gamma(x, y, -)|_{[0, d(x, y)]}$ is a geodesic connecting $x$ and $y$.
\end{definition}

\section{Ideal boundary}
\label{section:ideal-boundary}
For more details on the proof of the statements in
this section, see~\cite[Section 4]{fukaya2024boundaryfreeproductsmetric}
and \cite[Section 4]{FO17}.

Let $(X, d)$ be a coarsely convex space with a 
$(\lambda, k, E, C, \theta)$-coarsely convex bicombing $\Gamma$,
and let
$\texttt{rp}\Gamma$ be a reparametrised bicombing of $\Gamma$. We fix $o \in X$ as the base point of $X$.

\begin{definition}
 Set $k_1 = \lambda + k$, $D = 2(1 + E)k_1 + C$, 
 and $D_1 = 2D + 2$.
 We define a product 
 $(- \:|\: -)_o \colon X \times X \rightarrow \mathbb{R}_{\geq 0}$ by
 \[
   (x \:|\: y)_o \coloneqq 
   \min\{d(o, x), d(o, y), \sup\{t \in \mathbb{R}_{\geq 0} 
 \:|\: 
   d(\texttt{rp}\Gamma(o, x, t), \texttt{rp}\Gamma(o, y, t)) \leq D_1\}\}.
 \]
\end{definition}

\begin{lemma}
    Set $D_2 \coloneqq E(D_1 + 2k)$. For $x, y, z \in X$, we have
    \[
    (x \:|\: z)_o \geq \frac{1}{D_2}\min\{(x \:|\: y)_o, (y \:|\: z)_o\}.
    \]
\end{lemma}

Following \cite{FOY22}, we construct the ideal boundary of coarsely convex space $X$ by a set of sequences in $X$ tending to infinity. 

We let
    \[
    S^{\infty}_o X \coloneqq \{\{x_n\}_{n \in \mathbb{N}} \:|\: \{x_n\}_{n \in \mathbb{N}} \text{ is a sequence such that } (x_n \:|\: x_m)_o \rightarrow \infty \text{ as } n, m \rightarrow \infty\}, 
    \]
and we define the relation $\sim$ on $S^{\infty}_o X$ as follows. For every $\{x_n\}_{n \in \mathbb{N}}, \{y_n\}_{n \in \mathbb{N}} \in S^{\infty}_o X$, we define
$ \{x_n\}_{n \in \mathbb{N}} \sim \{y_n\}_{n \in \mathbb{N}} $
if and only if
\[
    (x_n \:|\: y_n)_o \rightarrow \infty \text{ as }n \rightarrow \infty.
\] 
Then the relation $\sim$ is an equivalence relation on $S^{\infty}_o X$.

\begin{definition}
 The ideal boundary of a coarsely convex space $X$, denote by $\partial_o X$ is the quotient of the set $S^{\infty}_o X$ 
 by the equivalence relation $\sim$. We denote by $\bar{X}$ the disjoint union 
 of $X$ and $\partial_o X$. Namely,
 \[
   \partial_o X \coloneqq S^{\infty}_o X/\sim, \quad 
   \bar{X} \coloneqq X \cup \partial_o X.
 \]
\end{definition}
For $x \in X$ and a sequence $\{x_n\}_{n \in \mathbb{N}}$ in $X$, we write $\{x_n\}_{n \in \mathbb{N}} \in x$ if $x_n = x$ for every $n \in \mathbb{N}$.
We extend $(- \:|\: -)_o \colon X \times X \times \rightarrow \mathbb{R}_{\geq 0}$
to the function $(- \:|\: -)_o \colon \bar{X} \times \bar{X} \rightarrow \mathbb{R}_{\geq 0}$ as follows
\[
(x \:|\: y)_o \coloneqq \sup\{\liminf_{n \rightarrow \infty}(x_n \:|\: y_n)_o \:|\: \{x_n\}_{n \in \mathbb{N}} \in x, \{y_n\}_{n \in \mathbb{N}} \in y\}
\]
for any $x, y \in \bar{X}$.

For $n \in \mathbb{N}$, let
\[
V_n \coloneqq \{(x, y) \in \bar{X} \times \bar{X} \:|\: (x \:|\: y)_o > n\} \cup \left\{(x, y) \in X \times X \:|\: d(x, y) < \frac{1}{n}\right\}.
\]
Then $\{V_n\}_{n \in \mathbb{N}}$ is a base of a metrizable uniformly 
topology on $\bar{X}$. 
For any $x \in \bar{X}$, we define $V_n[x] \subset \bar{X}$ as follows
\[
V_n[x] \coloneqq \{y \in \bar{X} \:|\: (x, y) \in V_n\}.
\]
Then the family $\{V_n[x]\}_{n \in \mathbb{N}}$ is a fundamental system of neighborhoods of $x$.

\begin{proposition}[{\cite[Proposition 4.19.]{FO17}}]
\label{idealboundary-compact-metric-space}
For sufficiently small $\epsilon>0$, 
 there exists a constant $K\geq 1$ depending on $D,\theta(0)$ and $\epsilon$,
 and there exists a metric $d_\epsilon$ on $\partial_o X$ such that,  
 for all $x,y\in \partial_o X$,
 \begin{align*}
  \frac{1}{K} e^{-\epsilon(x\mid y)}
  \leq d_\epsilon(x,y)\leq e^{-\epsilon(x\mid y)} 
 \end{align*} 
 Especially, $d_{\epsilon}$ is compatible with the topology of $\partial_o X$, and
 the diameter of $(\partial_o X, d_\epsilon)$ is
 less than or equal to 1.
\end{proposition}

We extend the reparametrised bicombing $\texttt{rp}\Gamma \colon X \times X \times \mathbb{R}_{\geq 0} \rightarrow X$ to $\texttt{rp}\bar{\Gamma} \colon X \times \bar{X} \times \mathbb{R}_{\geq 0} \rightarrow X$.

\begin{lemma}[{\cite[Lemma 4.8.]{fukaya2024boundaryfreeproductsmetric}}]
\label{lem:combing-at-infinity}
    Suppose that $X$ is proper. 
 Then there exists a map $\texttt{rp}\bar{\Gamma} \colon X \times \bar{X} \times \mathbb{R}_{\geq 0} \to X$ satisfying the following.

    \begin{enumerate}[label=(\roman*)]
        \item \label{item:combing-at-infinity--restriction}
              For $x, y \in X$, we have $\texttt{rp}\bar{\Gamma}(x, y,  -) = \texttt{rp}\Gamma(x, y, -)$.
        
        \item \label{item:combing-at-infinity--convergence}
             For $(o, x) \in X \times \partial_o X$, there exists a sequence $\{x_n\}_{n \in \mathbb{N}}$ in $X$ such that the sequence of maps $\{\texttt{rp}\Gamma(o, x_n,-)|_{\mathbb{N}} \colon \mathbb{N} \rightarrow X\}_{n \in \mathbb{N}}$ converges to $\texttt{rp}\bar{\Gamma}(o, x, -)|_{\mathbb{N}}$ pointwise. 
        In particular, if $\Gamma$ is a geodesic bicombing,
        the sequence of maps $\{\texttt{rp}\Gamma(o, x_n, -) \colon \mathbb{R}_{\geq 0} \rightarrow X\}_{n \in \mathbb{N}}$ converges to $\texttt{rp}\bar{\Gamma}(o, x, -)$ uniformly on compact sets.

        \item \label{item:combing-at-infinity--Gromov-product}
              For $(o, x) \in X \times \partial_o X$, we have
              \[
              (\text{rp}\bar{\Gamma}(o, x, t) \:|\: x)_o \rightarrow 
              \infty \:(t \rightarrow \infty).
              \]

        \item \label{item:combing-at-infinity--geodesic}
              For $(o, x) \in X \times \partial_o X$, the map $\texttt{rp}\bar{\Gamma}(o, x,-) \colon \mathbb{R}_{\geq 0} \rightarrow X$ is a $(\lambda, k_1)$-quasi geodesic, where $k_1 \coloneqq \lambda + k$.
              In particular, if $\Gamma$ is a geodesic bicombing, $\texttt{rp}\bar{\Gamma}(o, x, -) \colon \mathbb{R}_{\geq 0} \rightarrow X$ is a geodesic.
    \end{enumerate}
\end{lemma}

\begin{definition}
    We call the map $\texttt{rp}\bar{\Gamma}$ given in Lemma 3.4 an extended bicombing on $X \times \bar{X}$ corresponding to $\Gamma$. For $(o, x) \in X \times \partial_o X$, we abbreviate $\texttt{rp}\bar{\Gamma}(o, x,-)$ by $\gamma^x_o(-)$.
\end{definition}

\begin{lemma}[{\cite[Lemma 4.10.]{fukaya2024boundaryfreeproductsmetric}}]
    The extended bicombing $\texttt{rp}\bar{\Gamma}$ on $X \times \bar{X}$ satisfies the following
    \begin{enumerate}[label=(\roman*)]
        \item 
        Let $(o_1, x_1), (o_2, x_2) \in X \times \partial_o X$ and let $a, b \in \mathbb{R}_{\geq 0}$. Set ${x_1}' \coloneqq \texttt{rp}\bar{\Gamma}(o_1, x_1, a)$ and ${x_2}' \coloneqq \texttt{rp}\bar{\Gamma}(o_2, x_2, b)$. Then, for $c \in [0, 1]$, we have
        \[
        d(\texttt{rp}\bar{\Gamma}(o_1, x_1, ca), \texttt{rp}\bar{\Gamma}(o_2, x_2, cb)) \leq (1-c)Ed(o_1, o_2) + cEd({x_1}',{x_2}') + D
        \]

        \item 
        We define a non-decreasing function $\tilde{\theta}(t) \coloneqq \theta(t + 1) + 1$. Let $(o_1, x_1), (o_2, x_2) \in X \times \partial_o X$. Then for $t, s \in \mathbb{R}_{\geq 0}$, we have
        \[
        |t - s| \leq \tilde{\theta}(d(o_1, o_2) + d(\texttt{rp}\bar{\Gamma}(o_1, x_1, t), \texttt{rp}\bar{\Gamma}(o_2, x_2, s))).
        \]
    \end{enumerate}
\end{lemma}

\begin{definition}
    For $x, y \in \partial_o X$, we define
    \[
    (\gamma^x_o \:|\: \gamma^y_o)_o \coloneqq \sup\{t \in \mathbb{R}_{\geq 0} \:|\: d(\gamma^x_o(t), \gamma^y_o(t)) \leq D_1 \}.
    \]
    For $x \in \partial_o X$ and $p \in X$, we define
    \[
    (\gamma^x_o \:|\: p)_o \coloneqq \min\{d(o, p), \sup\{t \in \mathbb{R}_{\geq 0} \:|\: d(\Gamma(o, p, t), \gamma^x_o(t)) \leq D_1\}\}.
    \]
\end{definition}

\begin{lemma}[{\cite[Lemma 4.12.]{fukaya2024boundaryfreeproductsmetric}}]
\label{lem:gromov-product}
    There exists a constant $\Omega \geq 1$ depending on $\lambda$, $k$, $E$, $C$, $\theta(0)$ such that the following holds

    \begin{enumerate}[label=(\arabic*)]
        \item 
              For $x, y \in \partial_o X$, we have
        \[
        (\gamma^x_o \:|\: \gamma^y_o)_o \leq (x \:|\: y)_o \leq \Omega (\gamma^x_o \:|\: \gamma^y_o)_o.
        \]

        \item
             For $x, y, z \in \partial_o X$, we have 
        \[
        (\gamma^x_o \:|\: \gamma^y_o)_o \geq \Omega^{-1}\min\{(\gamma^x_o \:|\: \gamma^y_o)_o, (\gamma^y_o \:|\: \gamma^z_o)_o\}.
        \]

        \item
             For $x, y, z \in \bar{X}$, we have
        \[
        (x \:|\: z)_o \geq \Omega^{-1}\min\{(x \:|\: y)_o, (y \:|\: z)_o\}.
        \]

        \item
                   Let $x, y \in \partial_o X$. For all $t \in \mathbb{R}_{\geq 0}$ with $t \leq (\gamma^x_o \:|\: \gamma^y_o)$, we have
        \[
        d(\gamma^x_o(t), \gamma^y_o(t)) \leq \Omega.
        \]

        \item
                   Let $o \in X$ and let $x, y \in \bar{X}$. If $\texttt{rp}\bar{\Gamma}(o, x, a) = \texttt{rp}\bar{\Gamma}(o, y, b)$ for some $a, b \in \mathbb{R}_{\geq 0}$, then for all $t \in [0, \max\{a, b\}]$ we have
        \[
        d(\texttt{rp}\bar{\Gamma}(o, x, t), \texttt{rp}\bar{\Gamma}(o, y, t)) \leq \Omega.
        \]

        \item
             Let $o \in X$ and $x \in \partial_o X$. For $v \in X$ and $t \in [0, 1]$, we have
        \[
        (x \:|\: \Gamma(o, v, t))_o \geq \Omega^{-1}\min\{(x \:|\: v)_o, td(o, v)\}.
        \]
    \end{enumerate}
\end{lemma}

\section{Horoboundary}\label{section:horoboundary}

\subsection{Horoboundary}
Let $(X,d)$ be a proper metric space and fix $o \in X$ as the base point. 
Also, let $C(X)$ be the set of all $\R$-valued continuous functions on $X$, 
and equip with the topology of uniform convergence on compact sets. 
It is a standard fact that the space $C(X)$ is Hausdorff. %


We define $\phi \colon X \rightarrow C(X)$ 
as follows. For $x \in X$, let 
\[
 x \mapsto \phi_x \coloneqq d(-, x) - d(o, x).
\]

\begin{proposition}\label{prop:phi-is-inj-and-conti}
    The map $\phi$ is injective and continuous.
\end{proposition}

\begin{proof}
 The triangle inequality implies that $\phi_x (u) - \phi_y (u) \leq 2d(x, y)$, for all $u, x, y \in X$. The continuity of $\phi$ follows.
 
 Let $x$ and $y$ be distinct points in $X$ such that 
 $d(o, x) \geq d(o, y)$. 
 We have
 \begin{align*}
  \phi_y (x) - \phi_x (x) &= d(x, y) - d(o, y) - d(x, x) + d(o, x)\\
  &\geq d(x, y),
 \end{align*}
 which shows that $\phi_x$ and $\phi_y$ are distinct.
\end{proof}

\begin{definition}\label{def:horoboundary}
 Let $(X,d)$ be a proper metric space and fix $o \in X$ as the base point. 
 Let $\phi \colon X \rightarrow C(X)$ be the map defined above. 
 We define the horoboundary $\partial_h X$ of $(X,d)$ as follows
 \[
   \partial_h X \coloneqq \cl\phi(X) \setminus
   \phi(X).
 \]
 Here $\cl \phi(X)$ denotes the closure of $\phi(X)$ in $C(X)$. Since $C(X)$ with the topology of uniform convergence on compact sets is Hausdorff, the horoboundary $\partial_h X$ is Hausdorff.
\end{definition}

\begin{lemma}
    Let $(X, d)$ be a proper metric space.
    If the sequence of maps $\{\phi_{x_n}\}_{n \in \mathbb{N}}$ converges to
    $\xi \in \partial_h X$,
    then only finitely many of the points of 
    $\{x_n\}_{n \in \mathbb{N}}$ lie in any bounded subset of $X$.
\end{lemma}

\begin{proof}
    If there is an infinite number of points of 
    $\{x_n\}_{n \in \mathbb{N}}$ contained some closed ball 
    $\bar{B}(o, R)$, 
    we can take a subsequence 
    $\{x_{n(k)}\}_{k \in \mathbb{N}}$ of 
    $\{x_n\}_{n \in \mathbb{N}}$ such that
    $\{x_{n(k)}\}_{k \in \mathbb{N}}$ converges to some point 
    $\tilde{x} \in \bar{B}(o, R)$.
    By hypothesis, a sequence of maps 
    $\{\phi_{x_{n(k)}}\}_{k \in \mathbb{N}}$
    converges to $\xi \in \partial_h X$.
    On the other hand, by \cref{prop:phi-is-inj-and-conti},
    $\{\phi_{x_{n(k)}}\}_{k \in \mathbb{N}}$ converges to 
    $\phi_{\tilde{x}}$.
    Therefore, we have $\xi = \phi_{\tilde{x}}$.
    This contradicts to $\xi \notin \phi(X)$.
    
\end{proof}

\begin{proposition}
    The horoboundary does not depend on the choice of the base point.
\end{proposition}

\begin{proof}
    We take $o, o' \in X$ arbitrarily.
    We define $\phi \colon X \rightarrow C(X)$ and 
    $\phi' \colon X \rightarrow C(X)$ by
    \[
    x \mapsto \phi_x(-) \coloneqq d(-, x) - d(o, x)
    \text{$\quad$ and $\quad$}
    x \mapsto \phi'_x(-) \coloneqq d(-, x) - d(o', x).
    \]

    Then, we have
    $
    \phi'_x = \phi_x - \phi_x(o') 
    \text{$\quad$ and $\quad$}
    \phi_x = \phi'_x - \phi'_x(o)
    $
    for any $x \in X$.

 Let $\partial_h X \coloneqq \mathrm{cl}\phi(X) \setminus \phi(X)$
    and $\partial'_h X \coloneqq \mathrm{cl}\phi'(X) \setminus \phi'(X)$.
    We define $F \colon \partial_h X \rightarrow C(X)$ by
    \[
    \xi \mapsto \xi - \xi(o').
    \]
    We show $F(\xi) \in \partial'_h X$ 
    for any $\xi \in \partial_h X$.
    Let
    $\{x_n\}_{n \in \mathbb{N}} \subset X$ be a sequence such that
    $d(o, x_n) \rightarrow \infty$ as $n \rightarrow \infty$ and
    a sequence of maps $\{\phi_{x_n}\}_{n \in \mathbb{N}}$
    converges uniformly to $\xi$ on compact sets.
    For any $R > 0$, we have
    \begin{align*}
        \sup_{u \in \bar{B}(o, R)}|F(\xi)(u) - \phi'_{x_n}(u)| 
        &= \sup_{u \in \bar{B}(o, R)}
        |\{\xi(u) - \xi(o')\} - \{\phi_{x_n}(u) - \phi_{x_n}(o')\}|\\
        &\leq \sup_{u \in \bar{B}(o, R)}|\xi(u) - \phi_{x_n}(u)|
        + |\xi(o') - \phi_{x_n}(o')|\\
        &\rightarrow 0 \text{ as } n \rightarrow \infty.
    \end{align*}
    So we have $F(\xi) \in \cl \phi'(X)$. 
    We assume $F(\xi) \in \phi'(X)$. Then there is some $z \in X$ 
    such that 
    \[
    F(\xi) = \xi - \xi(o') = \phi'_z = \phi_z - \phi_z(o').
    \]
    Since $\xi(o) = \phi_z(o) = 0$, we have $\xi(o') = \phi_z(o')$.
    Thus, we get $\xi = \phi_z$, but this contradicts to 
    $\xi \notin \phi(X)$. 
    Therefore, $\xi' \notin \phi'(X)$ and 
    $\xi' = F(\xi) \in \partial'_h X$.
    Thus, the map
    $F\colon \partial_h X\to \partial'_h X$ is well-defined.

    Similarly, we define 
    $\tilde{F} \colon \partial'_h X \to \partial_h X$ by
    \[
    \xi' \mapsto \xi' - \xi(o).
    \]

    We show that 
    $\tilde{F}$ and $F$ are the inverse of each other.
    Since $\xi(o) = \xi'(o') = 0$, we have
    \begin{align*}
    \tilde{F} \circ F(\xi) 
    = \tilde{F}(\xi - \xi(o'))
    = (\xi - \xi(o')) - (\xi(o) - \xi'(o))
    = \xi\\
    F \circ \tilde{F}(\xi')
    = F(\xi' - \xi'(o))
    = (\xi' - \xi(o')) - (\xi'(o') - \xi(o'))
    = \xi'.
    \end{align*}

    We show $F \colon \partial_h X \rightarrow \partial'_h X$ is continuous.
    Let $\{\xi_m\}_{m \in \mathbb{N}} \subset \partial_h X$
    be a sequence of maps such that 
    $\xi_m$ converges uniformly to 
    some $\xi \in \partial_h X$ on compact sets. 
    For any $L > 0$, we have
    \begin{align*}
        \sup_{u \in \bar{B}(o, L)}|F(\xi_m)(u) - F(\xi)(u)|
        &= \sup_{u \in \bar{B}(o, L)}
        |\{\xi_m(u) - \xi_m(o')\} - \{\xi(u) - \xi(o')\}|\\
        &\leq \sup_{u \in \bar{B}(o, L)}|\xi_m(u) - \xi(u)|
        + |\xi_m(o') - \xi(o')|\\
        &\rightarrow 0 \text{ as } m \rightarrow \infty.
    \end{align*}
    Thus, a sequence of maps 
    $\{F(\xi_m)\}_{m \in \mathbb{N}}$ converges uniformly to
    $F(\xi)$ on compact sets.
    Therefore, $F \colon \partial_h X \rightarrow \partial'_h X$
    is continuous.
    In the same way, we can show 
    $\tilde{F} \colon \partial'_h X \rightarrow \partial_h X$
    is continuous.
    Hence, $\partial_h X$ and $\partial'_h X$ are homeomorphic.
\end{proof}

\begin{example}\label{ex:l-1-horoboundary}
 Let $(\mathbb{R}^2, l^{1})$ be the space $\mathbb{R}^2$ with $l^{1}$ metric. 
 Here, $l^{1}$ is defined as follows. For 
 $(x_1, y_1), \: (x_2, y_2) \in \mathbb{R}^{2}$
 \[
   l^{1}((x_1, y_1), (x_2, y_2)) = |x_1 - x_2| + |y_1 - y_2|.
 \]
 Then horoboundary $\partial_h (\mathbb{R}^{2}, l^{1})$ is homeomorphic to 
 $\partial [-\infty, \infty]^{2}$. 
\end{example}

\begin{proposition}
\label{prop:geodesic-sp-top-emb}
 If $X$ is a geodesic space, 
 then $\phi \colon X \rightarrow C(X)$ 
 is a topological embedding.
\end{proposition}

\begin{proof}
 First, we show the following claim.
 \begin{claim0}
  Let $\{x_n\}_{n \in \mathbb{N}}$ be a sequence on $X$ such that 
  $d(o, x_n) \rightarrow \infty$ and 
  $\phi_{x_n}$ converges to 
  $\xi \in \cl\phi(X)$ uniformly on compact sets. 
  Then we show that there is no 
  $y \in X$ such that $\xi = \phi_y$. 
 \end{claim0} 

 We prove this by contradiction. 
 That is, suppose there is some $y \in X$ and
 $\phi_{x_n}$ converges uniformly on a compact set to $\phi_y$. 
 
 Since
 $d(y, x_n) \geq d(o, x_n) - d(o, y)$ 
 from the triangle inequality,
 $d(y, x_n)$ diverges to infinity. 
 We let $\gamma_n \colon [0, d(y, x_n)]\rightarrow X$ 
 be a geodesic connecting $y$ and $x_n$. 
 That is,
 $\gamma_n(0) = y$ and $\gamma_n(d(y, x_n)) = x_n$. 
 We take $x'_n$
 to be a point on the image $\gamma_n([0, d(y, x_n)])$ of the geodesic $\gamma_n$, 
 satisfying $d(y, x'_n) = d(o, y) + 1$. 
 For any $n \in \mathbb{N}$, $x'_n$ is contained in 
 $\bar{B}(y, d(o, y)+1)$. 
 Since $\phi_{x_n}$ converges uniformly to $\phi_y$ on compact sets, 
 the following holds.
 \begin{align*}
  |\phi_y(x'_n) - \phi_{x_n}(x'_n)| 
  &\leq \sup_{u \in \bar{B}(y, d(o, y)+1)}|\phi_y(u) - \phi_{x_n}(u)|\\
  &\rightarrow 0 \text{ as } n \rightarrow \infty.
 \end{align*}
 
 For any $n \in \mathbb{N}$, 
 since $x_n, x'_n, y \in X$ lie on the geodesic $\gamma_n$, we have
 \begin{align*}
  \phi_y(x'_n) - \phi_{x_n}(x'_n) 
  &= \{d(x'_n, y) - d(o, y)\} - \{d(x'_n, x_n) - d(o, x_n)\}\\
  &= 1 - d(x'_n, x_n) + d(o, x_n)\\
  &\geq 1 - d(x'_n, x_n) + d(x_n, y) - d(o, y)\\
  &= 1 + d(y, x'_n) - d(o, y)\\
  &= 2.
 \end{align*}

 This contradicts the results above, so the claim holds.
 Now we show that if a sequence of maps 
 $\{\phi_{x_n}\}_{n \in \mathbb{N}}$ 
 converges uniformly to $\phi_x$ $(x \in X)$ on compact sets,
 a sequence 
 $\{x_n\}_{n \in \mathbb{N}}$
 converges to $x$.

 By the claim, the set $\{x_n\}$ is bounded. Set
 \[
 R \coloneqq \max\{d(o, x), \, \sup_{n \in \mathbb{N}}d(o, x_n)\}.
 \]
 Since
 $\phi_{x_n}$ converges pointwise to $\phi_x$, 
 we have
    \begin{align*}
        |\phi_{x_n}(x) - \phi_x(x)| &= 
     |(d(x, x_n) - d(o, x_n)) - (d(x, x) - d(o, x))|\\
        &= |d(x, x_n) - d(o, x_n) + d(o, x)|\\
        &\rightarrow 0 \text{ as } n \rightarrow \infty.
    \end{align*}
    
     On the other hand, $\phi_{x_n}$ converges uniformly to $\phi_x$ on compact sets, so we obtain
    \begin{align*}
        |\phi_{x_n}(x_n) - \phi_{x}(x_n)| &\leq \sup_{u \in \bar{B}(o, R)}|\phi_{x_n}(u) - \phi_x(u)|\\
        & \rightarrow 0 \text{ as } n \rightarrow  \infty.
    \end{align*}
    
    Thus, we obtain the following
    \begin{align*}
        |\phi_{x_n}(x_n) - \phi_{x}(x_n)| 
        &= |d(x_n, x_n) - d(o, x_n) - d(x_n, x) + d(o, x)| \\
        &= |-d(o, x_n) - d(x_n, x) + d(o, x)|\\
        &= |d(o, x_n) + d(x_n, x) - d(o, x)|\\
        & \rightarrow 0 \text{ as } n \rightarrow \infty.
    \end{align*}
    By triangle inequality, we have
    \begin{align*}
        2d(x_n, x) &\leq |d(x, x_n) - d(o, x_n) + d(o, x)|\\
        &\phantom{+ \quad}
        +|d(o, x_n) + d(x_n, x) - d(o, x)|\\
        & \rightarrow 0 \text{ as } n \rightarrow \infty.
    \end{align*}
\end{proof}


\subsection{Reduced horoboundary}

\begin{definition}
    
Let $(X,d)$ be a proper metric space and let $\partial_h X$ be the horoboundary.
 We define that two functions $\xi$ and $\eta$ in $\partial_h X$ are equivalent,
 denoted by $\xi \sim \eta$, if 
    \[
    \sup_{x \in X}|\xi(x) - \eta(x)| < \infty.
    \]

 This determines an equivalence relation on 
 $\partial_h X$. 
 A \emph{reduced horoboundary of $X$}, 
 denoted by $\partial_h X/\sim$, is the quotient of 
 $\partial_h X$ by the equivalence relation $\sim$ with the quotient topology.
\end{definition}

In general, a geodesic Gromov hyperbolic space does not necessarily coincide with its horoboundary and its ideal boundary. 

On the other hand, in geodesic Gromov hyperbolic spaces,
Webster and Winchester \cite{WW03} showed the following.

\begin{theorem}[{\cite[Theorem 4.5.]{WW03}}]\label{thm:WW}
    Let $(X, d)$ be a proper geodesic Gromov hyperbolic space with the base point $o \in X$. 
    Let $\partial_o X$ be an ideal boundary with respect to the base point $o$ and let $\partial_h X$ be a horoboundary of $X$.
    Then there is a natural continuous quotient map from $\partial_h X$
    onto $\partial_o X$.
\end{theorem}

It is well known that the reduced horoboundary coincides with the ideal boundary in proper geodesic Gromov hyperbolic spaces as a corollary of \cref{thm:WW}.

The same result does not hold for coarsely convex spaces.
For example, we have seen that all normed spaces are geodesic coarsely convex spaces in \cref{ex:norm-is-cc}.
There are examples such as the following. 
Here, we remark that an ideal boundary of a proper coarsely convex space is compact and metrizable

\begin{example}
 \label{ex:reduced-horo-non-Hausdorff}
 The reduced horoboundary of $(\mathbb{R}^2, l^1)$ is not Hausdorff.
\end{example}

\section{Open cone and horoboundary}
\label{section:open-cone}
\begin{definition}
    Let $(Y,d_Y)$ be a compact metric space with diameter less than or equal to 1. 
 Let the open cone over $Y$ be the following quotient topological space.
    \[
    \mathcal{O}Y = ([0, \infty) \times Y)/(\{0\} \times Y)
    \]
    For $(t, y) \in [0, \infty] \times Y$, we denote by $tx$ the point in $\mathcal{O}Y$ represented by $(t, y)$. We define a metric $d_{\mathcal{O}Y}$ on $\mathcal{O}Y$ by
    \[
    d_{\mathcal{O}Y}(tx, sy) \coloneqq |t-s| + \min\{t,s\}d_Y(x, y).
    \]
    Set $o \coloneqq 0y \in \mathcal{O}Y$ as the base point of 
    $\mathcal{O}Y$.
\end{definition}

\begin{proposition}
 Let $(Y,d_Y)$ be a compact metric space with diameter less than 
 or equal to 1 and $\mathcal{O}Y$ be an open cone over $Y$. Then $\mathcal{O}Y$ is a proper metric space.
\end{proposition}

\begin{proof}
Let $K \subset \mathcal{O}Y$ be a closed bounded subset.
There exists $R\geq 0$ such that $K$ is contained in 
$[0,R]\times Y/(\{0\}\times Y)$. Since $[0,R]\times Y$ is compact,
the quotient $[0,R]\times Y/(\{0\}\times Y)$ is compact.
Therefore, $K$ is compact.
\end{proof}

As in the following example, $\mathcal{O}Y$ is not necessarily a geodesic space.

\begin{example}
   We consider a two-points set $Y=\{a, b\}$ and define 
   $d \colon Y \times Y \rightarrow \mathbb{R}_{\geq 0}$ 
   to satisfy $d(a,a) = 0, d(b,b) = 0, d(a,b) = d(b,a) = 1$. 
   Then $d$ is a metric on $Y$, 
   and $(Y,d)$ is a compact metric 
   space with diameter less than or equal to 1. 
   
   Then $\mathcal{O}Y$ is not a geodesic space. 
   We prove this by contradiction.
   
   We assume there exists some geodesic
   $\gamma \colon [0, 1] \rightarrow \mathcal{O}Y$
   such that $\gamma(0) = 1a$ and $\gamma(1) = 1b$.
   In this case, there is a map 
   $f \colon [0,1] \rightarrow \mathbb{R}_{\geq 0}$ and 
   a map $g \colon [0,1] \rightarrow Y$
   such that $\gamma(t) = f(t)g(t)$ for any $t \in [0,1]$.
   
   We show $f \colon [0,1]\rightarrow \mathbb{R}_{\geq 0}$ is a continuous function. 
   By the continuity of $\gamma \colon [0,1]\rightarrow \mathcal{O}Y$, 
   for any $t_1 \in [0,1]$ and any $\epsilon >0$, 
   there exists some $\delta >0$ and 
   for any $t \in [0,1]$ with $|t - t_1| < \delta$,
   we have
   \[
       d_{\mathcal{O}Y}(\gamma(t),\gamma(t_1)) = |f(t) - f(t_1)| + \min\{f(t), f(t_1)\}d(g(t), g(t_1)) < \epsilon. 
   \]
   So we have $|f(t) - f(t_1)| <  \epsilon$.
   
   By the compactness of $[0,1]$ and the continuity of $f$, 
   there exists minimum value $m \coloneqq \min_{t \in [0,1]}f(t)$. Then, the minimum value $m$ is equal to 0.
   We prove this by contradiction.
   Since $f(t) \geq 0$ for all $t \in [0,1]$, 
   we may assume that $m > 0$. 
   From the continuity of 
   $\gamma \colon [0,1]\rightarrow \mathcal{O}Y$, 
   for any ${t_1}' \in [0,1]$ and any $\epsilon' > 0$, 
   there exists some $\delta' > 0$ and for any $t' \in [0,1]$ with $|t' - {t_1}'| < \delta$, 
   we have
   \[
       d_{\mathcal{O}Y}(\gamma(t'), \gamma({t_1}'))= |f(t') - f({t_1}')| + \min\{f(t'), f({t_1}')\}d_Y(g(t'), g({t_1}')) < m\epsilon. 
   \]
    So we have 
    \[
    md_Y(g(t'), g({t_1}')) < m\epsilon.
    \]
    By the hypothesis of $m > 0$, we get 
    \[
    d_Y(g(t'), g({t_1}')) < \epsilon.
    \]
    
   Therefore, $g \colon [0,1]\rightarrow Y$ 
   is a continuous map. 
   Since 
   $[0,1]$ is connected and $Y$ is discrete,
   the map $g \colon [0,1] \rightarrow Y$ is a constant map. 
   We assume that $g(t)=1a$ for any $t \in [0,1]$. 
   Then, we have
   \begin{align*}
       0 = d_{\mathcal{O}Y}(\gamma(1), \gamma(1)) &= d_{\mathcal{O}Y}(1b, f(1)g(1)) = d_{\mathcal{O}Y}(1b, f(1)a) \\
       &= |1 - f(1)| + \min\{1, f(1)\}d_Y(a,b)\\
       &\geq \min\{1, m\} > 0.
   \end{align*}
   This is contradiction. 
   So we have $m = 0$. 
   In particular, there exists some $t_0 \in [0,1]$ such that $f(t_0) = 0$. It follows that $\gamma(t_0) = f(t_0)g(t_0) = o$. 
   Since $1b \neq o$, we have $t_0 \neq 1$. 
   On the other hand, since $\gamma \colon [0,1]\rightarrow \mathcal{O}Y$ is a geodesic, 
   we have $t_0 
   = d_{\mathcal{O}Y}(\gamma(0), \gamma(t_0)) 
   = d_{\mathcal{O}Y}(1a, o) = 1$. 
   This is a contradiction.
\end{example}

We will show that $\mathcal{O}Y$ can be topologically embedded in 
$C(\mathcal{O}Y)$.
As stated in the above,  $\mathcal{O}Y$ is not necessarily geodesic spaces.
So, we cannot apply directly \cref{prop:geodesic-sp-top-emb}.

\begin{proposition}
\label{prop:open-cone-top-emb}
 Let $(Y, d_Y)$ be a compact metric space with diameter less than or equal to 1 
 and $\mathcal{O}Y$ be an open cone over $Y$. 
    Let $C(\mathcal{O}Y)$ be the space of continuous functions on $\mathcal{O}Y$ 
    equipped with the topology of uniform convergence on compact sets. 
    For each $sx \in \mathcal{O}Y$,
    we consider the function 
    $\phi_{sx} \colon \mathcal{O}Y \rightarrow \mathbb{R}$ as following,
    \begin{align*}
        \phi_{sx} (ty) &\coloneqq d_{\mathcal{O}Y} (ty, sx) - d_{\mathcal{O}Y} (o, sx)\\
        &= d_{\mathcal{O}Y} (ty, sx) - s
    \end{align*}
    Then the map 
    $\phi \colon \mathcal{O}Y \rightarrow C(\mathcal{O}Y), sx \mapsto \phi_{sx}$ is a topological embedding.
\end{proposition}

\begin{proof}
    Let $\{z_n\}_{n \in \mathbb{N}} = \{t_n y_n\}_{n \in \mathbb{N}}, (t_n \in \mathbb{R}_{\geq 0},\: y_n \in Y)$ 
    be a sequence in $\mathcal{O}Y$ 
    escaping to infinity, that is $t_n \rightarrow \infty$.

 We claim that no subsequence of $\phi_{z_n}$ 
 converges to a function $\phi_z$ with $z \coloneqq ty \in \mathcal{O}Y$. 
Suppose contrarily, there exists a subsequence of $\phi_{z_n}$ which converges to $\phi_z$ with $z=ty \in \mathcal{O}Y$.
By replacing the subsequence, we suppose that  $\phi_{z_n}$
converges to $\phi_z$.

First we consider the case $t=0$. 

We have $\phi_z(1y) = \phi_o(1y) = 1$.
 On the other hand, for any $n \in \mathbb{N}$ with 
 $t_n\geq 1$, 
 we have $\phi_{z_n}(1y) = -1 + d_Y(y_n, y) \leq 0$.
 This contradicts that $\phi_{z_n}$ converges pointwise to $\phi_o$.
 
    
 Next we consider the case $t \neq 0$. 

    Let $R>0$ be a constant that satisfies $R>t+1$. 
    Set $x = (t+1)y \in \mathcal{O}Y $.

 Since the topology of $C(\mathcal{O}Y)$ is given by the uniform convergence on compact sets,
 $\phi_{z_n}$ converges uniformly on $\bar{B}(0, R)$ to $\phi_z$. 
 By taking $n \in \mathbb{N}$ large enough, 
 we have
 \begin{align*}
  |\phi_{z_n}(z) - \phi_z(z)| 
  &= |(d_{\mathcal{O}Y}(z_n, z)- d_{\mathcal{O}Y}(0, z_n)) 
  - (d_{\mathcal{O}Y}(z, z)- d_{\mathcal{O}Y}(0, z))| \\
  &= |(|t_n - t| + \min\{t_n, t\}d_Y(y_n, y) - t_n)-(-t)|\\
  &= td_Y(y_n, y) \rightarrow 0 \text{ as } n \rightarrow \infty.
 \end{align*}
 Since $t \neq 0$, we have $d_Y(y_n, y) \rightarrow 0 
 \text{ as } n \rightarrow \infty$. 
 On the other hand, for $n\in \N$ with $t_n\geq t+1$,
 \begin{align*}
  &|\phi_{z_n}(x) - \phi_z(x)|\\
  &= |(|t_n - (t+1)| + \min\{t_n,t\}d_Y(y_n, y) - t_n)-
  (1+td_Y(y,y)-t)|\\
  &= |-2+ td_Y(y_n, y)|
  \rightarrow 2 \neq 0 \text{ as } n \rightarrow \infty.
 \end{align*}
This contradicts the fact that $\phi_{z_n}$ converges pointwise to $\phi_z$. 
 This completes the proof of the claim. 

 The rest of the proof can be shown as in \cref{prop:geodesic-sp-top-emb}.

\end{proof}

\begin{proposition}\label{prop:open-cone-convergence}
    Let $(Y,d_Y)$ be a compact metric space with diameter less than or equal to 1, and let
    $\mathcal{O}Y$ be an open cone over $Y$. 
    Let $\{t_n y_n\}_{n \in \mathbb{N}}$ be a sequence satisfying 
    $t_n \rightarrow \infty$ on $\mathcal{O}Y$. 
    If $\phi_{t_n y_n} \in C(\mathcal{O}Y)$ converges uniformly on compact sets to 
    $\xi \in \mathrm{cl}\{\phi_sx \: | \: sx \in \mathcal{O}Y \}$, 
    then $y_n$ converges to some point on $Y$.
\end{proposition}

\begin{proof}
 Since $Y$ is compact, it is enough to show that 
 $\{y_n\}_{n \in \mathbb{N}}$ is a Cauchy sequence.
 
    Let $m$ and $n$ be arbitrary natural numbers. 
    Since $\phi_{t_n y_n}$ converges uniformly to 
    $\xi \in \partial_h \mathcal{O}Y$  on compact sets, 
    we have
    \begin{align*}
        |\phi_{t_m y_m}(1y_n) - \phi_{t_n y_n}(1y_n)| &\leq |\phi_{t_m y_m}(1y_n) - \xi(1y_n)| + |\phi_{t_n y_n}(1y_n) - \xi(1y_n)|\\
        &\leq \sup_{u \in \bar{B}(o, 1)}|\phi_{t_m y_m}(u) - \xi(u)| + \sup_{u \in \bar{B}(o, 1)}|\phi_{t_n y_n}(u) - \xi(u)|\\
        &\rightarrow 0 \text{ as } m, n \rightarrow \infty.
    \end{align*}
    On the other hand, by taking $m$ and $n$ large enough, 
    we have
    \begin{align*}
        |\phi_{t_m y_m}(1y_n) - \phi_{t_n y_n}(1y_n)| &= |d_{\mathcal{O}Y}(1y_n, t_m y_m) - d_{\mathcal{O}Y}(o, t_m y_m) - d_{\mathcal{O}Y}(1y_n, t_n y_n) + d_{\mathcal{O}Y}(o, t_n y_n)|\\
        &= |(t_m - 1) + d_Y(y_m, y_n) - t_m - (t_n -1) - d_Y(y_n, y_n) + t_n|\\
        &= d_Y(y_m, y_n).
    \end{align*}
    So we have $d_Y(y_m, y_n) \rightarrow 0 \text{ as }  m, n \rightarrow \infty$.
\end{proof}

\begin{proposition}\label{prop:horoboundary-of-open-cone}
     Let $(Y,d_Y)$ be a compact metric space with diameter less than or equal to 1.
     Let $(\mathcal{O}Y, d_{\mathcal{O}Y})$ be an open cone over $Y$.
     Then $Y$ and $\partial_h \mathcal{O}Y$ are homeomorphic.
\end{proposition}

\begin{proof}
    We define 
    $F \colon Y \rightarrow C(\mathcal{O}Y)$
    as follows,
    \[
    y \mapsto (F(y)(sx) \coloneqq -s + sd(x, y)).
    \]
    For any $y \in Y$, we show $F(y) \in \partial_h \mathcal{O}Y$.

 We observe that a sequence $\{\phi_{ny}\}_{n\in \N}$
 converges to $F(y)$.
 We fix $R > 0$. For all $n > R$, we have
    \begin{align*}
        &\sup_{sx \in \bar{B}(o, R)}|\phi_{n y}(sx) - F(y)(sx)|\\
        &= \sup_{sx \in \bar{B}(o, R)}
        |\{d_{\mathcal{O}Y}(sx, ny) - d_{\mathcal{O}Y}(o, ny)\}
        - \{-s + sd_Y(x, y)\}|\\
        &= \sup_{sx \in \bar{B}(o, R)}
        |\{n - s + sd_Y(x, y) - n\} - \{-s + sd_Y(x, y)\}|\\
        &= 0.
    \end{align*}
    
    Thus, the sequence of maps 
    $\{\phi_{n y}\}_{n \in \mathbb{N}}$
    converges uniformly to $F(y)$ on compact sets.
    Since the sequence 
    $\{n y\}_{n \in \mathbb{N}}$ is unbounded,
    we have 
    $F(y) = \lim_{n \rightarrow \infty}\phi_{n y} \in \partial_h \mathcal{O}Y$
    by \cref{prop:open-cone-top-emb}.
    
    We show $F \colon Y \rightarrow \partial_h \mathcal{O}Y$ is continuous.

    We take a sequence 
    $\{y_n\}_{n \in \mathbb{N}} \subset Y$
    such that 
    $y_n$ converges to some $y \in Y$.
    We take $L > 0$ arbitrarily.
    Then we have
    \begin{align*}
        &\sup_{sx \in \bar{B}(o, L)}|F(y_n)(sx) - F(y)(sx)|\\
        &= \sup_{sx \in \bar{B}(o, L)}
        |\{-s + sd_Y(x, y_n)\} - \{-s + sd_Y(x, y)\}|\\
        &\leq \sup_{sx \in \bar{B}(o, L)}sd_Y(y_n, y)\\
        &\leq Ld_Y(y_n, y) \rightarrow 0 \text{ as } n \rightarrow \infty.
    \end{align*}
    Thus, a sequence of maps
    $\{F(y_n)\}_{n \in \mathbb{N}}$
    converges uniformly to $F(y)$ on compact sets,
    so $F \colon Y \rightarrow \partial_h \mathcal{O}Y$ is continuous.
    
    We show $F \colon Y \rightarrow \partial_h \mathcal{O}Y$ is injective.
    For any $y, z \in Y$, we assume $F(y) = F(z)$.
    Then we have
    \[
    F(y)(1z) = -1 + d_Y(y, z) = -1 = F(z)(1z).
    \]
    Therefore, we have $d_Y(y, z) = 0$, that is, 
    $F \colon Y \rightarrow \partial_h \mathcal{O}Y$ is injective.
    
    We show $F \colon Y \rightarrow \partial_h \mathcal{O}Y$ is surjective.
    We take $\xi \in \partial_h \mathcal{O}Y$ arbitrarily.
    Then, there is a sequence 
    $\{t_n y_n\}_{n \in \mathbb{N}} \subset \mathcal{O}Y$
    such that 
    $t_n \rightarrow \infty$ as $n \rightarrow \infty$
    and 
    a sequence of maps 
    $\{\phi_{t_n y_n}\}_{n \in \mathbb{N}}$
    converges uniformly to $\xi$ on compact sets.
    By \cref{prop:open-cone-convergence}, the sequence 
    $\{y_n\}_{n \in \mathbb{N}}$
    converges to some point $y \in Y$.
    For any $sx \in \mathcal{O}Y$, we have
    \begin{align*}
        \xi(sx) &= \lim_{n \rightarrow \infty}\phi_{t_n y_n}(sx)\\
                &= \lim_{n \rightarrow \infty}
                    d_{\mathcal{O}Y}(sx, t_n y_n) - d_{\mathcal{O}Y}(o, t_n y_n)\\
                &= \lim_{n \rightarrow \infty}
                    t_n - s + sd_Y(x, y_n) - t_n\\
                &= -s + sd(x, y)\\
                &=F(y)(sx).
    \end{align*}
    Therefore, we have $\xi = F(y)$, that is, 
    $F \colon Y \rightarrow \partial_h \mathcal{O}Y$ is surjective.
    Since $Y$ is compact and $\partial_h \mathcal{O}Y$ is Hausdorff,
    $F \colon Y \rightarrow \partial_h \mathcal{O}Y$ is a homeomorphism.
\end{proof}

By \cref{idealboundary-compact-metric-space}, an 
ideal boundary of a proper coarsely convex space is a
compact metric space with its diameter if less than or equal to 1.
Applying \cref{prop:horoboundary-of-open-cone}, 
we obtain \cref{thm:horoboundary-of-opencone}.


\section{Cone metric and horoboundary}
\label{section:cone-metric}
 Andreev defined a cone metric for Busemann spaces \cite{And18}. 
 In this section, we construct a cone metric for coarsely convex spaces. 
 Then we construct the horoboundary using the cone metric.

\subsection{Cone metric}

For $t\in \R$, we denote by $\lfloor t \rfloor$ the greatest integer
less than or equal to $t$.

\begin{definition}\label{def:qg-cone-metric}
Let $(X, d)$ be a $(\lambda, k, E, C, \theta)$-coarsely convex space with a base point $o \in X$. Let $\Gamma$ be a $(\lambda, k, E, C, \theta)$-coarsely convex bicombing on $X$.
 For given points $x, y \in X$.
    The cone metric $d_c(x, y)$ is defined by
        \begin{align*}
            d_c(x, y) \coloneqq &|d(o, x) - d(o, y)| \\
            & \: + d(\texttt{rp}\Gamma(o, x, \lfloor\min\{d(o, x), d(o, y)\}\rfloor), \texttt{rp}\Gamma(o, y, \lfloor\min\{d(o, x), d(o, y)\}\rfloor)).
        \end{align*}
    By the definition of the cone metric, 
    for any $x \in X$, we have $d_c(o, x) = d(o, x)$.
    Also, the cone metric is non-negative, symmetric.
\end{definition}

\begin{remark}
 If $(X, d)$ is a geodesic $(E, C)$-coarsely convex space with a geodesic $(E, C)$-coarsely convex bicombing $\Gamma$ and a base point $o \in X$, we define cone metric as follows
 \begin{align*}
d_c(x, y) \coloneqq &|d(o, x) - d(o, y)| \\
            & \: + d(\texttt{rp}\Gamma(o, x, \min\{d(o, x), d(o, y)\}), \texttt{rp}\Gamma(o, y, \min\{d(o, x), d(o, y)\}))
 \end{align*}
 for any $x, y \in X$. 
 From \cref{ex:Busemann-is-gcc},
 a Busemann space is a geodesic $(1, 0)$-coarsely convex space.
 Then, the cone metric defined above coincides with the definition given by Andreev in \cite{And18}.
 Furthermore, \cref{def:qg-cone-metric} is an extension of the cone metric defined above to quasi-geodesic spaces.
\end{remark}

\begin{lemma}\label{lem:qg-triangle-inequality}
    Let $(X, d)$ be a $(\lambda, k, E, C, \theta)$-coarsely convex space
    with $(\lambda, k, E, C, \theta)$-coarsely convex bicombing $\Gamma$.
    For any three points $x, y, z \in X$, we have 
    \[
    d_c(x, z) \leq  \lambda Ed_c(x, y) + \lambda Ed_c(y, z) + 4\lambda + 2k + C.
    \]
\end{lemma}

\begin{proof}
    We consider three points $x, y, z \in X$ satisfy $d(o, x) \leq d(o, y) \leq d(o, z)$.
    Since $\Gamma$ is a $(\lambda, k, E, C, \theta)$-coarsely convex bicombing, we have
    \begin{align*}
        &d(\texttt{rp}\Gamma(o, y, \lfloor d(o, x) \rfloor), \texttt{rp}\Gamma(o, z, \lfloor d(o, x) \rfloor))\\
        &=d\left(\Gamma \left(o, y, \frac{\lfloor d(o, x) \rfloor}{d(o, y)}\right),
        \Gamma \left(o, z, \frac{\lfloor d(o, x) \rfloor}{d(o, z)}\right)\right)\\
        &\leq \frac{\lfloor d(o, x) \rfloor}{\lfloor d(o, y) \rfloor} E
        d\left(\Gamma \left(o, y, \frac{\lfloor d(o, y) \rfloor}{d(o, y)}\right), \Gamma \left(o, z, \frac{\lfloor d(o, y) \rfloor}{d(o, z)}\right)\right) + C\\
        &\leq Ed(\texttt{rp}\Gamma(o, y, \lfloor d(o, y) \rfloor), \texttt{rp}\Gamma(o, z, \lfloor d(o, y) \rfloor)) + C.
    \end{align*}

     Thus, we have the following
     \begin{align*}
         d_c(x, y) &= d(o, y) - d(o, x) + d(\texttt{rp}\Gamma(o, x, \lfloor d(o, x) \rfloor), \texttt{rp}\Gamma(o, y, \lfloor d(o, x) \rfloor))\\
         &\leq d(o, z) - d(o, x) + d(\texttt{rp}\Gamma(o, x, \lfloor d(o, x) \rfloor), \texttt{rp}\Gamma(o, z, \lfloor d(o, x) \rfloor))\\
         &\phantom{ \leq \quad} + d(\texttt{rp}\Gamma(o, y, \lfloor d(o, x) \rfloor), \texttt{rp}\Gamma(o, z, \lfloor d(o, x) \rfloor))\\
         &= d_c(x, z) + d(\texttt{rp}\Gamma(o, y, \lfloor d(o, x) \rfloor), \texttt{rp}\Gamma(o, z, \lfloor d(o, x) \rfloor))\\
         &\leq d_c(x, z) + d(o, z) - d(o, y) + Ed(\texttt{rp}\Gamma(o, y, \lfloor d(o, y) \rfloor), \texttt{rp}\Gamma(o, z, \lfloor d(o, y) \rfloor)) + C\\
         &\leq Ed_c(x, z) + Ed_c(y, z) + C.
     \end{align*}

     Similarly, we have
     \begin{align*}
         d_c(x, z) &= d(o, z) - d(o, x) + d(\texttt{rp}\Gamma(o, x, \lfloor d(o, x) \rfloor), \texttt{rp}\Gamma(o, z, \lfloor d(o, x) \rfloor))\\
         &\leq d(o, z) - d(o, x) + d(o, y) - d(o, y) + d(\texttt{rp}\Gamma(o, x, \lfloor d(o, x) \rfloor), \texttt{rp}\Gamma(o, y, \lfloor d(o, x) \rfloor))\\
         &\phantom{ \leq \quad} + d(\texttt{rp}\Gamma(o, y, \lfloor d(o, x) \rfloor), \texttt{rp}\Gamma(o, z, \lfloor d(o, x) \rfloor))\\
         &= d_c(x, y) + d(o, z) - d(o, y) + d(\texttt{rp}\Gamma(o, y, \lfloor d(o, x) \rfloor), \texttt{rp}\Gamma(o, z, \lfloor d(o, x) \rfloor))\\
         &\leq d_c(x, y) + d(o, z) - d(o, y) + Ed(\texttt{rp}\Gamma(o, y, \lfloor d(o, y) \rfloor), \texttt{rp}\Gamma(o, z, \lfloor d(o, y) \rfloor)) + C\\
         &\leq Ed_c(x, y) + Ed_c(y, z) + C.
     \end{align*}

     Since $\Gamma$ is a $(\lambda, k)$-quasi geodesic bicombing, we have
     \begin{align*}
     d(\texttt{rp}\Gamma(o, y, \lfloor d(o, x) \rfloor), \texttt{rp}\Gamma(o, y, \lfloor d(o, y) \rfloor))
     &\leq \lambda \left(\frac{\lfloor d(o, y) \rfloor}{d(o, y)} - \frac{\lfloor d(o, x) \rfloor}{d(o, y)}\right)d(o, y) + k\\
     &\leq \lambda(d(o, y) - d(o, x)) + 2\lambda + k.
     \end{align*}
     Similarly, we have
      \begin{align*}
     d(\texttt{rp}\Gamma(o, z, \lfloor d(o, x) \rfloor), \texttt{rp}\Gamma(o, z, \lfloor d(o, y) \rfloor))
     &\leq \lambda \left(\frac{\lfloor d(o, y) \rfloor}{d(o, z)} - \frac{\lfloor d(o, x) \rfloor}{d(o, z)}\right)d(o, z) + k\\
     &\leq \lambda(d(o, y) - d(o, x)) + 2\lambda + k.
     \end{align*}

     So we have
     \begin{align*}
         d_c(y, z) &= d(o, z) - d(o, y) + d(\texttt{rp}\Gamma(o, y, \lfloor d(o, y) \rfloor), \texttt{rp}\Gamma(o, z, \lfloor d(o, y) \rfloor))\\
         &\leq d(o, z) - d(o, y) + d(\texttt{rp}\Gamma(o, y, \lfloor d(o, x) \rfloor), \texttt{rp}\Gamma(o, y, \lfloor d(o, y) \rfloor))\\
         &\phantom{\leq \quad}
          + d(\texttt{rp}\Gamma(o, y, \lfloor d(o, x) \rfloor), \texttt{rp}\Gamma(o, x, \lfloor d(o, x) \rfloor))\\
         &\phantom{\leq \quad}
          + d(\texttt{rp}\Gamma(o, x, \lfloor d(o, x) \rfloor), \texttt{rp}\Gamma(o, z, \lfloor d(o, x) \rfloor))\\
         &\phantom{\leq \quad}
         + d(\texttt{rp}\Gamma(o, z, \lfloor d(o, x) \rfloor), \texttt{rp}\Gamma(o, z, \lfloor d(o, y) \rfloor))\\
         &\leq d(o, z) - d(o, y) + \lambda(d(o, y) - d(o, x)) + 2\lambda + k\\
         &\phantom{\leq \quad}
         + d(\texttt{rp}\Gamma(o, y, \lfloor d(o, x) \rfloor), \texttt{rp}\Gamma(o, x, \lfloor d(o, x) \rfloor))\\
         &\phantom{\leq \quad}
          + d(\texttt{rp}\Gamma(o, x, \lfloor d(o, x) \rfloor), \texttt{rp}\Gamma(o, z, \lfloor d(o, x) \rfloor))
          + \lambda(d(o, y) - d(o, x)) + 2\lambda + k\\
        &\leq \lambda(d(o, z) - d(o, y)) + \lambda(d(o, y) - d(o, x)) + 2\lambda + k\\
         &\phantom{\leq \quad}
         + d(\texttt{rp}\Gamma(o, y, \lfloor d(o, x) \rfloor), \texttt{rp}\Gamma(o, x, \lfloor d(o, x) \rfloor))\\
         &\phantom{\leq \quad}
          + d(\texttt{rp}\Gamma(o, x, \lfloor d(o, x) \rfloor), \texttt{rp}\Gamma(o, z, \lfloor d(o, x) \rfloor))
          + \lambda(d(o, y) - d(o, x)) + 2\lambda + k\\
        &\leq \lambda d_c(x, y) + \lambda d_c(x, z) + 4\lambda + 2k.
     \end{align*}
\end{proof}

\begin{lemma}\label{lem:qg-bi-Lip}
    Let $(X, d)$ be a $(\lambda, k, E, C, \theta)$-coarsely convex space and let $d_c$ be a cone metric on $X$.
    For any $x, y \in X$, we have
    \[
    d_c(x, y) \leq (\lambda + 2)d(x, y) + 2\lambda + 2k.
    \]
\end{lemma}

\begin{proof}
    We do not lose generality by assuming $d(o, x) \leq d(o, y)$.
    Since $\Gamma$ is a $(\lambda, k)$-quasi geodesic bicombing, we have
    \begin{align*}
        &d(\texttt{rp}\Gamma(o, x, \lfloor d(o, x) \rfloor), \texttt{rp}\Gamma(o, y, \lfloor d(o, x) \rfloor))\\
        &\leq d(\texttt{rp}\Gamma(o, x, \lfloor d(o, x) \rfloor), x)
        + d(x, y) + d(y, \texttt{rp}\Gamma(o, y, \lfloor d(o, x) \rfloor))\\
        &= d\left(\Gamma\left(o, x, \frac{\lfloor d(o, x) \rfloor}{d(o, x)}\right), \Gamma(o, x, 1)\right) + d(x, y) + d\left(\Gamma(o, y, 1), \Gamma\left(o, y, \frac{\lfloor d(o, x) \rfloor}{d(o, y)}\right)\right)\\
        &\leq \lambda + k + d(x, y) + \lambda(d(o, y) - d(o, x) + 1) + k\\
        &\leq (\lambda + 1)d(x, y) + 2\lambda + 2k.
    \end{align*}
    So we have
    \begin{align*}
    d_c(x, y) &= d(o, y) - d(o, x) + d(\texttt{rp}\Gamma(o, x, \lfloor d(o, x) \rfloor), \texttt{rp}\Gamma(o, y, \lfloor d(o, x) \rfloor))\\
    &\leq (\lambda + 2)d(x, y) + 2\lambda + 2k.
    \end{align*}
\end{proof}

\subsection{Horoboundary}
Now, we consider horoboundary for the cone metric, and 
study the relation with the ideal boundary.
Let $B(X)$ be the set of all $\R$-valued functions on $X$, and equip with the topology
of uniform convergence on compact sets.
We define $\psi \colon X \rightarrow B(X)$ as follows. For $x \in X$, let
\begin{align*}
    x \mapsto \psi_x &\coloneqq d_c(-, x) - d_c(o, x)\\
                     &\coloneqq d_c(-, x) - d(o, x).
\end{align*}

\begin{definition}\label{def:qg-cone-metric-horoboundary}
  Let $(X,d)$ be a proper $(\lambda, k, E, C, \theta)$-coarsely convex space and let 
  $B_b(X)$ be the set of all $\R$-valued bounded functions on $X$. 
  Let $\psi \colon X \rightarrow B(X)$ be the map defined above. 
  Let
  \[
  B_{b, \lambda, E}(X) \coloneqq 
  \left\{
f \in B(X) 
  \relmiddle|
     \frac{1}{\lambda E}\psi_x - \beta \leq f \leq \lambda E\psi_x + \beta
     , \: \exists x \in X, \exists\beta \in B_b(X)
     \right\}.
     \]
  We define the \emph{horoboundary of $X$ with cone metric}, denoted by $\partial^c_h X$ as follows
     \[
     \partial^c_h X \coloneqq \cl\psi(X) \setminus B_{b,\lambda, E}(X).
     \]
 We also say that $\partial^c_h X$ is a horoboundary of $(X,d_c)$.
 Here $\cl\psi(X)$ denotes the closure of $\psi(X)$ in $B(X)$.

  We call a function which belongs to $\partial^c_h X$ a 
  \emph{horofunction with cone metric}, or we simply call it \emph{horofunction} if there is no risk of misunderstanding.
 \end{definition}

 \begin{lemma}
 \label{lem:qg-seq-diverges}
    Let $(X,d)$ be a proper $(\lambda, k, E, C, \theta)$-coarsely convex space. 
    If the sequence $\psi_{x_n}$ converges to $\xi \in \partial^c_h X$, 
    then only finitely many of the points $x_n$ lie in any bounded subset of $X$. 
\end{lemma}

\begin{proof}
    We assume that infinitely many 
    $x_n$ contained in some ball $\bar{B}(o, R)$ where $R > 0$. 

    Since $X$ is proper, $\bar{B}(o,R)$ is compact.  
    By taking a subsequence, we can assume that 
    $\{x_n\}_{n \in \mathbb{N}}$ converges to $\tilde{x} \in \bar{B}(o, R)$ as $n \rightarrow \infty$.
    By hypothesis, $\psi_{x_n}$ 
    converges pointwise to $\xi$ on $X$.
    By \cref{lem:qg-triangle-inequality} and \cref{lem:qg-bi-Lip}, we have 
    \begin{align*}
        \psi_{x_n}(x) - \lambda E\psi_{\tilde{x}}(x) 
        &= \{d_c(x, x_n) - d_c(o, x_n)\} - \lambda E\{d_c(x, \tilde{x}) - d_c(o, \tilde{x})\}\\
        &= \{d_c(x, x_n) - \lambda Ed_c(x, \tilde{x})\} + \{\lambda Ed_c(o, \tilde{x}) - d_c(o, x_n)\}\\
        &\leq \lambda Ed_c(x_n, \tilde{x}) +4\lambda + 2k + C + \{\lambda Ed(o, \tilde{x}) - d(o, x_n)\}\\
       &\leq \lambda E\{(\lambda + 2)d(x_n, \tilde{x}) + 2\lambda + 2k\}
       + 4\lambda + 2k + C \\
       &\phantom{\leq \quad}
       + \{\lambda Ed(o, \tilde{x}) - d(o, x_n)\}
    \end{align*}
    for all $x \in X$.
    Since $d(x_n, \tilde{x}) \rightarrow 0$ and 
    $\psi_{x_n}(x) \rightarrow \xi(x)$
    for all $x \in X$,  we get
    \begin{align*}
    \xi(x) - \lambda E\psi_{\tilde{x}}(x) 
    &= \lim_{n \rightarrow \infty}\psi_{x_n}(x) - \lambda E\psi_{\tilde{x}}(x)\\
    &\leq \lim_{n \rightarrow \infty}\lambda E\{(\lambda + 2)d(x_n, \tilde{x}) + 2\lambda + 2k\}
       + 4\lambda + 2k + C \\
     &\phantom{\leq \quad \quad \quad}
       + \{\lambda Ed(o, \tilde{x}) - d(o, x_n)\}\\
    &= (\lambda E - 1)d(o, \tilde{x}) + 2\lambda^2 E + 2\lambda kE + 4\lambda + 2k + C\\
    &\leq (\lambda E - 1)R + 2\lambda^2 E + 2\lambda kE + 4\lambda + 2k + C.
    \end{align*}
    So we have 
    \[
    \xi(x) \leq \lambda E\psi_{\tilde{x}}(x) + (\lambda E - 1)R + 2\lambda^2 E + 2\lambda kE + 4\lambda + 2k + C
    \]
    for any $x \in X$.

    Similarly, we have
    \[
    \psi_{\tilde{x}}(x) - \lambda E\xi(x) \leq (\lambda E - 1)R + 2\lambda^2 E + 2\lambda kE + 4\lambda + 2k + C
    \]
    for any $x \in X$, that is,
    \[
    \frac{1}{\lambda E}\psi_{\tilde{x}}(x) - \frac{1}{\lambda E}\{(\lambda E - 1)R + 2\lambda^2 E + 2\lambda kE + 4\lambda + 2k + C\} \leq \xi(x)
    \]
    for any $x \in X$.
    This contradicts to $\xi \in \partial^c_h X$.
\end{proof}

In the rest of this section, let $(X, d)$ be a proper metric space
with a $(\lambda, k, E, C, \theta)$-coarsely convex bicombing $\Gamma$,
fix the base point $o \in X$ 
and let $\partial_o X$ be the ideal boundary of $X$ with respect to the base point $o$. 
We construct a continuous map from $\partial^c_h X$ to $\partial_o X$.

\begin{lemma}\label{lem:qg-const-geodesic}
 Let $\xi \in \partial^c_h X$ and let $\{x_n\}_{n \in \mathbb{N}}$ on 
 $X$ be a sequence such that  
    $\psi_{x_n}$ converges uniformly to $\xi$ on compact sets. 
\begin{enumerate}[label=(\arabic*)]
 \item \label{item:qg-Gamma-x_n-R-converge}
       The sequence $\{x_n\}_{n \in \mathbb{N}}$ belongs to $S^{\infty}_o X$.
       That is, 
       $(x_n \:|\: x_m)_o \rightarrow \infty$
       as $n, m \rightarrow \infty$.
 \item \label{item:qg-Gamma-x_n-R-converge-independet}
       Let $\{y_n\}_{n \in \mathbb{N}}$ be another sequence on $X$ such that 
       $\psi_{y_n}$ converges uniformly to $\xi$ on compact sets. Then we have 
       $(x_n \:|\: y_n)_o \rightarrow \infty$
       as $n \rightarrow \infty$.
\end{enumerate}
\end{lemma}

\begin{proof}
    We take $R > 0$ arbitrarily. Set $ \tilde{x}_n = \texttt{rp}\Gamma(o, x_n, R)$.
    By the definition of $\Gamma$, for any $n \in \mathbb{N}$, we have
 \[
 d(o, \tilde{x}_n) = d\left(o,\texttt{rp}\Gamma(o, x_n, R) 
 \right)
 \leq \lambda R + k .
 \]
 So we have $\tilde{x}_n \in \bar{B}(o, \lambda R + k)$ for any $n \in \mathbb{N}$.
 Since $X$ is proper, $\bar{B}(o, \lambda R + k)$ is compact.

 Since $\psi_{x_n}$ converges uniformly to $\xi$ on $\bar{B}(o, \lambda R + k)$, we have
 \begin{align*}
        |\psi_{x_n}(\tilde{x}_n) - \psi_{x_m}(\tilde{x}_n)| &\leq |\psi_{x_n}(\tilde{x}_n) - \xi(\tilde{x}_n)| + |\psi_{x_m}(\tilde{x}_n) - \xi(\tilde{x}_n)|\\
        &\leq \sup_{u \in \bar{B}(o, \lambda R + k)}|\psi_{x_n}(u) - \xi(u)| + \sup_{u \in \bar{B}(o, \lambda R + k)}|\psi_{x_m}(u) - \xi(u)|\\
        &\rightarrow 0 \text{ as } n, m \rightarrow \infty.
    \end{align*}
On the other hand, for any $n, m \in \mathbb{N}$ with 
$d(o, x_n), d(o, x_m) > R$, we have
\begin{align}\label{eq:qg-psi-metric}
     &|\psi_{x_n}(\tilde{x}_n) - \psi_{x_m}(\tilde{x}_n)|\\
        &= |\{d_c(\tilde{x}_n, x_n) - d(o, x_n)\} - \{d_c(\tilde{x}_n, x_m) - d(o, x_m)\}|\notag \\
        &= |\{d(o, x_n) - d(o, \tilde{x}_n) + d(\texttt{rp}\Gamma(o, \tilde{x}_n, \lfloor d(o, \tilde{x}_n) \rfloor), \texttt{rp}\Gamma(o, x_n, \lfloor d(o, \tilde{x}_n) \rfloor)) - d(o, x_n)\}\notag \\
        &\phantom{ {=} \quad}
        - \{d(o, x_m) - d(o, \tilde{x}_n) + d(\texttt{rp}\Gamma(o, \tilde{x}_n, \lfloor d(o, \tilde{x}_n) \rfloor), \texttt{rp}\Gamma(o, x_m, \lfloor d(o, \tilde{x}_n) \rfloor)) - d(o, x_m)\}|\notag \\
        &= |d(\texttt{rp}\Gamma(o, \tilde{x}_n, \lfloor d(o, \tilde{x}_n) \rfloor), \texttt{rp}\Gamma(o, x_n, \lfloor d(o, \tilde{x}_n) \rfloor))\notag \\
        &\phantom{ {=} \quad}
        - 
        d(\texttt{rp}\Gamma(o, \tilde{x}_n, \lfloor d(o, \tilde{x}_n) \rfloor), \texttt{rp}\Gamma(o, x_m, \lfloor d(o, \tilde{x}_n) \rfloor))|\notag.
\end{align}
Since $\Gamma$ is a $(\lambda, k, E, C, \theta)$-coarsely convex bicombing, we have
\begin{align*}
    |R - d(o, \tilde{x}_n)|
    &=
    \left|\frac{R}{d(o, x_n)}d(o, x_n) - d(o, \tilde{x}_n) 
    \right|\\
    &\leq \theta \left( d \left(\Gamma \left(o, x_n, \frac{R}{d(o, x_n)}\right), \Gamma(o, \tilde{x}_n, 1)\right) \right)\\
    &= \theta(0).
\end{align*}
So we have
\begin{align}\label{eq:qg-close-of-qg}
    &d(\texttt{rp}\Gamma(o, \tilde{x}_n, \lfloor d(o, \tilde{x}_n) \rfloor), \texttt{rp}\Gamma(o, x_n, \lfloor d(o, \tilde{x}_n) \rfloor))\\
    &\leq d(\texttt{rp}\Gamma(o, \tilde{x}_n, \lfloor d(o, \tilde{x}_n) \rfloor), \texttt{rp}\Gamma(o, \tilde{x}_n, d(o, \tilde{x}_n)))
    + d(\texttt{rp}\Gamma(o, \tilde{x}_n, d(o, \tilde{x}_n)), \texttt{rp}\Gamma(o, x_n, R))\notag \\
    &\phantom{\leq \quad}
    + d(\texttt{rp}\Gamma(o, x_n, R), \texttt{rp}\Gamma(o, x_n, d(o, \tilde{x}_n)))
    + d(\texttt{rp}\Gamma(o, x_n, d(o, \tilde{x}_n)), \texttt{rp}\Gamma(o, x_n, \lfloor d(o, \tilde{x}_n) \rfloor))\notag \\
    &\leq \lambda + k + \lambda  |R - d(o, \tilde{x}_n)| + k 
    + \lambda + k \notag \\
    &\leq \lambda(\theta(0) + 2) + 3k \notag.
\end{align}
By \cref{eq:qg-psi-metric} and \cref{eq:qg-close-of-qg}, there is some $N \in \mathbb{N}$ such that we have
\begin{align*}
    &d(\texttt{rp}\Gamma(o, \tilde{x}_n, \lfloor d(o, \tilde{x}_n) \rfloor), \texttt{rp}\Gamma(o, x_m, \lfloor d(o, \tilde{x}_n) \rfloor))\\
    &\leq d(\texttt{rp}\Gamma(o, \tilde{x}_n, \lfloor d(o, \tilde{x}_n) \rfloor), \texttt{rp}\Gamma(o, x_n, \lfloor d(o, \tilde{x}_n) \rfloor)) + 1\\
    &\leq \lambda(\theta(0) + 2) + 3k + 1
\end{align*}
for any $n, m \in \mathbb{N}$ with $n, m \geq N$.
Therefore, we have
\begin{align*}
    &d(\texttt{rp}\Gamma(o, x_n, R), \texttt{rp}\Gamma(o, x_m, R))\\
    &\leq d(\texttt{rp}\Gamma(o, x_n, R), \texttt{rp}\Gamma(o, x_n, d(o, \tilde{x}_n)))
    + d(\texttt{rp}\Gamma(o, x_n, d(o, \tilde{x}_n)),
    \texttt{rp}\Gamma(o, x_n, \lfloor d(o, \tilde{x}_n)\rfloor))\\
    &\phantom{\leq \quad}
    + d( \texttt{rp}\Gamma(o, x_n, \lfloor d(o, \tilde{x}_n)\rfloor),  \texttt{rp}\Gamma(o, \tilde{x}_n, \lfloor d(o, \tilde{x}_n)\rfloor))\\
    &\phantom{\leq \quad}
    + d(\texttt{rp}\Gamma(o, \tilde{x}_n, \lfloor d(o, \tilde{x}_n) \rfloor), \texttt{rp}\Gamma(o, x_m, \lfloor d(o, \tilde{x}_n) \rfloor))\\
    &\phantom{\leq \quad}
    + d(\texttt{rp}\Gamma(o, x_m, \lfloor d(o, \tilde{x}_n) \rfloor), \texttt{rp}\Gamma(o, x_m, d(o, \tilde{x}_n)))\\
    &\phantom{\leq \quad}
    + d(\texttt{rp}\Gamma(o, x_m, d(o, \tilde{x}_n)), \texttt{rp}\Gamma(o, x_m, R))\\
    &\leq \lambda \theta(0) + k + \lambda + k + \lambda(\theta(0) + 2) + 3k +  
    \lambda(\theta(0) + 2) + 3k + 1 + \lambda + k + \lambda \theta(0) + k\\
    &= \lambda(4\theta(0) + 6) + 10k + 1
\end{align*}
for any $n, m \in \mathbb{N}$ with $n, m \geq N$.

Set $L \coloneq \lambda(4\theta(0) + 6) + 10k + 1$.
Note that $L > 1$. Then, for any $R > 0$, there is some 
$N \in \mathbb{N}$ such that for any $n, m \in \mathbb{N}$ 
with $n, m \geq N$, we have
\[
d(\texttt{rp}\Gamma(o, x_n, R), \texttt{rp}\Gamma(o, x_m, R)) \leq L.
\]

We show, for any $s \in \mathbb{R}_{\geq 0}$, there is some $N' \in \mathbb{N}$ such that for any $n, m \in \mathbb{N}$ with $n, m \geq N'$, we have 
$(x_n \:|\: x_m)_o \geq s$.
From above, there is some $N' \in \mathbb{N}$ such that for any $n, m \in \mathbb{N}$ with $n, m \geq N'$, we have
\[
d(\texttt{rp}\Gamma(o, x_n, ELs), \texttt{rp}\Gamma(o, x_m, ELs)) \leq L.
\]
Since $\Gamma$ is a $(\lambda, k, E, C, \theta)$-coarsely convex bicombing, we have
\begin{align*}
    d(\texttt{rp}\Gamma(o, x_n, s), \texttt{rp}\Gamma(o, x_m, s))
    &= d\left(\texttt{rp}\Gamma\left(o, x_n,\frac{1}{EL} EL s\right), \texttt{rp}\Gamma\left(o, x_m, \frac{1}{EL} ELs\right)\right)\\
    &\leq \frac{1}{EL}Ed(\texttt{rp}\Gamma(o, x_n, ELs), \texttt{rp}\Gamma(o, x_m, ELs)) + C\\
    &\leq 1 + C \leq D_1.
\end{align*}
So we have
\[
    (x_n \:|\: x_m)_o 
    = \sup\{t \in \mathbb{R}_{\geq 0} \:|\: d(\texttt{rp}\Gamma(o, x_n, t), \texttt{rp}\Gamma(o, x_m, t)) \leq D_1\}
    \geq s
\]
for any $n, m \in \mathbb{N}$ with $n, m \geq N'$.
This completes the proof of \cref{item:qg-Gamma-x_n-R-converge}.

Now we will show \cref{item:qg-Gamma-x_n-R-converge-independet}.
Since both $\psi_{x_n}$ and $\psi_{y_n}$  converge uniformly to 
 $\xi$ on $\bar{B}(o, \lambda R + k)$, we have
    \begin{align*}
        |\psi_{x_n}(\tilde{x}_n) - \psi_{y_n}(\tilde{x}_n)| 
        &\leq |\psi_{x_n}(\tilde{x}_n) - \xi(\tilde{x}_n)| + |\psi_{y_n}(\tilde{x}_n) - \xi(\tilde{x}_n)|\\
        &\leq \sup_{u \in \bar{B}(o, \lambda R + k)}|\psi_{x_n}(u) - \xi(u)| + \sup_{u \in \bar{B}(o, \lambda R + k)}|\psi_{y_n}(u) - \xi(u)|\\
        &\rightarrow 0 \text{ as } n \rightarrow \infty.
    \end{align*}
On the other hand, for any $n \in \mathbb{N}$ with 
$d(o, x_n), d(o, y_n) > R$, we have
\begin{align}\label{eq:qg-psi-metric-well-defined}
     &|\psi_{x_n}(\tilde{x}_n) - \psi_{y_n}(\tilde{x}_n)|\\
        &= |\{d_c(\tilde{x}_n, x_n) - d(o, x_n)\} - \{d_c(\tilde{x}_n, y_n) - d(o, y_n)\}|\notag \\
        &= |\{d(o, x_n) - d(o, \tilde{x}_n) + d(\texttt{rp}\Gamma(o, \tilde{x}_n, \lfloor d(o, \tilde{x}_n) \rfloor), \texttt{rp}\Gamma(o, x_n, \lfloor d(o, \tilde{x}_n) \rfloor)) - d(o, x_n)\}\notag \\
        &\phantom{ {=} \quad}
        - \{d(o, y_n) - d(o, \tilde{x}_n) + d(\texttt{rp}\Gamma(o, \tilde{x}_n, \lfloor d(o, \tilde{x}_n) \rfloor), \texttt{rp}\Gamma(o, y_n, \lfloor d(o, \tilde{x}_n) \rfloor)) - d(o, y_n)\}|\notag \\
        &= |d(\texttt{rp}\Gamma(o, \tilde{x}_n, \lfloor d(o, \tilde{x}_n) \rfloor), \texttt{rp}\Gamma(o, x_n, \lfloor d(o, \tilde{x}_n) \rfloor))\notag \\
        &\phantom{ {=} \quad}
        - 
        d(\texttt{rp}\Gamma(o, \tilde{x}_n, \lfloor d(o, \tilde{x}_n) \rfloor), \texttt{rp}\Gamma(o, y_n, \lfloor d(o, \tilde{x}_n) \rfloor))|\notag.
\end{align}
By \cref{eq:qg-close-of-qg} and \cref{eq:qg-psi-metric-well-defined}, there is some $N \in \mathbb{N}$ such that we have
\begin{align*}
    &d(\texttt{rp}\Gamma(o, \tilde{x}_n, \lfloor d(o, \tilde{x}_n) \rfloor), \texttt{rp}\Gamma(o, y_n, \lfloor d(o, \tilde{x}_n) \rfloor))\\
    &\leq d(\texttt{rp}\Gamma(o, \tilde{x}_n, \lfloor d(o, \tilde{x}_n) \rfloor), \texttt{rp}\Gamma(o, x_n, \lfloor d(o, \tilde{x}_n) \rfloor)) + 1\\
    &\leq \lambda(\theta(0) + 2) + 3k + 1
\end{align*}
for any $n \in \mathbb{N}$ with $n \geq N$. 
Therefore, we have 
\begin{align*}
    &d(\texttt{rp}\Gamma(o, x_n, R), \texttt{rp}\Gamma(o, y_n, R))\\
    &\leq d(\texttt{rp}\Gamma(o, x_n, R), \texttt{rp}\Gamma(o, x_n, d(o, \tilde{x}_n)))
    + d(\texttt{rp}\Gamma(o, x_n, d(o, \tilde{x}_n)),
    \texttt{rp}\Gamma(o, x_n, \lfloor d(o, \tilde{x}_n)\rfloor))\\
    &\phantom{\leq \quad}
    + d( \texttt{rp}\Gamma(o, x_n, \lfloor d(o, \tilde{x}_n)\rfloor),  \texttt{rp}\Gamma(o, \tilde{x}_n, \lfloor d(o, \tilde{x}_n)\rfloor))\\
    &\phantom{\leq \quad}
    + d(\texttt{rp}\Gamma(o, \tilde{x}_n, \lfloor d(o, \tilde{x}_n) \rfloor), \texttt{rp}\Gamma(o, y_n, \lfloor d(o, \tilde{x}_n) \rfloor))\\
    &\phantom{\leq \quad}
    + d(\texttt{rp}\Gamma(o, y_n, \lfloor d(o, \tilde{x}_n) \rfloor), \texttt{rp}\Gamma(o, y_n, d(o, \tilde{x}_n)))\\
    &\phantom{\leq \quad}
    + d(\texttt{rp}\Gamma(o, y_n, d(o, \tilde{x}_n)), \texttt{rp}\Gamma(o, y_n, R))\\
    &\leq \lambda \theta(0) + k + \lambda + k + \lambda(\theta(0) + 2) + 3k +  
    \lambda(\theta(0) + 2) + 3k + 1 + \lambda + k + \lambda \theta(0) + k\\
    &= \lambda(4\theta(0) + 6) + 10k + 1 = L
\end{align*}
for any $n \in \mathbb{N}$ with $n \geq N$.
The rest of the proof of \cref{item:qg-Gamma-x_n-R-converge-independet} can be shown as in \cref{item:qg-Gamma-x_n-R-converge}.
\end{proof}
By \cref{lem:qg-const-geodesic}, we define 
$\mathrm{Pr} \colon \partial^c_h X \rightarrow \partial_o X$ as follows.
For any $\xi \in \partial^c_h X$, 
there is a sequence $\{x_n\}_{n \in \mathbb{N}} \in S^{\infty}_o X$ 
such that $\psi_{x_n}$ converges uniformly to $\xi$ on compact sets.
Let $x \in \partial_o X$ be the equivalence class of 
the sequence $\{x_n\}_{n \in \mathbb{N}}$.
Then we define $\mathrm{Pr}(\xi) \coloneqq x$.
By \cref{item:qg-Gamma-x_n-R-converge-independet} of \cref{lem:qg-const-geodesic}, $\mathrm{Pr}(\xi)$ does not depend on the choice of a sequence 
$\{x_n\}_{n \in \mathbb{N}}$ such that 
$\xi = \lim_{n \rightarrow \infty}\psi_{x_n}$.

\begin{proposition}
\label{prop:qg-Pr-is-continuous}
    $\mathrm{Pr} \colon \partial^c_h X \rightarrow \partial_o X$ is continuous.
\end{proposition}

\begin{proof}
     We take $\xi \in \partial^c_h X$ 
    arbitrarily and assume that the sequence of maps
    $\{\xi_m\}_{m \in \mathbb{N}}$ in 
    $\partial^c_h X$ converges uniformly to $\xi$ on compact sets. 

     We take $\{x_n\}_{n \in \mathbb{N}}$ to be 
    a sequence on $X$ such that 
    $\{\psi_{x_n}\}_{n \in \mathbb{N}}$ 
    converges uniformly to $\xi$ on compact sets.
    By \cref{lem:qg-const-geodesic}, 
    the sequence $\{x_n\}_{n \in \mathbb{N}}$ belongs to $S^{\infty}_o X$,
    and we denote by $x$ the equivalence class of 
    $\{x_n\}_{n \in \mathbb{N}}$ on $\partial_o X$.
    By the definition of the map
    $\mathrm{Pr}$, we get $\mathrm{Pr}(\xi) = x$.

    Similarly,  we take 
    $\{y_{m,l}\}_{l \in \mathbb{N}}$ 
    to be a sequence on $X$ such that 
    $\{\psi_{y_{m,l}}\}_{l \in \mathbb{N}}$ 
    converges uniformly to $\xi_m$ on compact sets for each $m \in \mathbb{N}$. 
    \cref{lem:qg-const-geodesic}, the sequence 
    $\{y_{m, l}\}_{l \in \mathbb{N}}$ belongs to $S^{\infty}_o X$
    each $l \in \mathbb{N}$, and we denote by $y_m$ 
    the equivalence class of 
    $\{y_{m, l}\}_{l \in \mathbb{N}}$ on $\partial_o X$.
    By the definition of the map
    $\mathrm{Pr}$, we get $\mathrm{Pr}(\xi_m) = y_m$ for any $m \in \mathbb{N}$.

    Here, we show 
    $\lim_{m \rightarrow \infty}(x \:|\: y_m)_o = \infty$.
    Since we have 
    $\lim_{n \rightarrow \infty}(x \:|\: x_n)_o = \infty$ and
    $\lim_{l \rightarrow \infty}(y_{m, l} \:|\: y_m)_o = \infty$,
    it is enough to show
    $\lim_{m \rightarrow \infty}\liminf_{n,l \rightarrow \infty}(x_n \:|\: y_{m, l})_o = \infty$.

    We take $R > 0$ arbitrarily. 
    Set $\tilde{x}_n \coloneqq \texttt{rp}\Gamma(o, x_n, R)$.
    We have $\tilde{x}_n \in \bar{B}(o, \lambda R + k)$ for any $n \in \mathbb{N}$.
    Since $\psi_{x_n}$ converges uniformly to $\xi$ on compact sets and $\psi_{y_{m, l}}$ converges uniformly to $\xi_m$ on compact sets for each $m \in \mathbb{N}$, we have
    \begin{align*}
        &\liminf_{n, l \rightarrow \infty}
        |\psi_{x_n}(\tilde{x}_n) - \psi_{y_{m, l}}(\tilde{x}_n)|\\
        &\leq \liminf_{n,l \rightarrow \infty}|\psi_{x_n}(\tilde{x}_n) - \xi(\tilde{x}_n)| + |\xi(\tilde{x}_n) - \xi_m(\tilde{x}_n)| + |\xi_m(\tilde{x}_n) - \psi_{y_{m, l}}(\tilde{x}_n)|\\
        &\leq \lim_{n \rightarrow \infty}\sup_{u \in \bar{B}(o, \lambda R + k)}|\psi_{x_n}(u) - \xi(u)| 
        + \sup_{u \in \bar{B}(o, \lambda R + k)}|\xi(u) - \xi_m(u)|\\
        &\phantom{ \lim_{n \rightarrow \infty} \quad}
        + \lim_{l \rightarrow \infty}\sup_{u \in \bar{B}(o, \lambda R + k)}|\xi_m(u) - \psi_{y_{m, l}}(u)|\\
        &= \sup_{u \in \bar{B}(o, \lambda R + k)}|\xi(u) - \xi_m(u)|.
    \end{align*}
    Since $\xi_m$ converges uniformly to $\xi$ on compact sets, 
    for any $R > 0$, there is some $M > 0$ such that for any $m \in \mathbb{N}$ with $m \geq M$ and for a sufficiently small $\epsilon$, we have
    \[
    \sup_{u \in \bar{B}(o, \lambda R + k)}|\xi(u) - \xi_m(u)| \leq 1 - \epsilon.
    \]
    We fix $m \in \mathbb{N}$ with $m \geq M$.
    Then, we have
    \[
    \liminf_{n, l \rightarrow \infty}
        |\psi_{x_n}(\tilde{x}_n) - \psi_{y_{m, l}}(\tilde{x}_n)|
    \leq 1 - \epsilon.
    \]
    There is some $N_m \in \mathbb{N}$ such that for any 
    $i \in \mathbb{N}$ with $i \geq N_m$, 
    there are some $n, l \in \mathbb{N}$ such that $n, l \geq i$ and we have
    $|\psi_{x_n}(\tilde{x}_n) - \psi_{y_{m, l}}(\tilde{x}_n)| < 1$.

    On the other hand, by \cref{lem:qg-seq-diverges}, for large enough $n \in \mathbb{N}$, we have $d(o, x_n) > \lambda R + k$.
    Similarly, for large enough $l \in \mathbb{N}$, we have
    $d(o, y_{m, l}) > \lambda R + k$. So we have
    \begin{align*}
        &|\psi_{x_n}(\tilde{x}_n) - \psi_{y_{m, l}}(\tilde{x}_n)|\\
        &= |\{d(o, x_n) - d(o, \tilde{x}_n) + d(\texttt{rp}\Gamma(o, x_n, \lfloor d(o, \tilde{x}_n) \rfloor), \texttt{rp}\Gamma(o, \tilde{x}_n, \lfloor d(o, \tilde{x}_n) \rfloor)) - d(o, x_n) \}\\
        &\phantom{ {=} \quad}
        - \{d(o, y_{m, l}) - d(o, \tilde{x}_n) + 
        d(\texttt{rp}\Gamma(o, y_{m, l}, \lfloor d(o, \tilde{x}_n) \rfloor), \texttt{rp}\Gamma(o, \tilde{x}_n, \lfloor d(o, \tilde{x}_n) \rfloor) - d(o, y_{m, l})\}|\\
        &=|d(\texttt{rp}\Gamma(o, x_n, \lfloor d(o, \tilde{x}_n) \rfloor), \texttt{rp}\Gamma(o, \tilde{x}_n, \lfloor d(o, \tilde{x}_n) \rfloor))\\
        &\phantom{ {=} \quad}
        - d(\texttt{rp}\Gamma(o, y_{m, l}, \lfloor d(o, \tilde{x}_n) \rfloor), \texttt{rp}\Gamma(o, \tilde{x}_n, \lfloor d(o, \tilde{x}_n) \rfloor))|
    \end{align*}
    for large enough $n, l \in \mathbb{N}$.
    By the proof of \cref{lem:qg-const-geodesic}, we have
    \[
    d(\texttt{rp}\Gamma(o, x_n, \lfloor d(o, \tilde{x}_n) \rfloor), \texttt{rp}\Gamma(o, \tilde{x}_n, \lfloor d(o, \tilde{x}_n) \rfloor)) \leq \lambda(\theta(0) + 2) + 3k
    \]
    for each $n \in \mathbb{N}$.
    Thus, for any $i \in \mathbb{N}$ with $i \geq N_m$, there is some $n, l \in \mathbb{N}$ such that $n, l \geq i$ and we have 
    \[
    d(\texttt{rp}\Gamma(o, y_{m, l}, \lfloor d(o, \tilde{x}_n) \rfloor), \texttt{rp}\Gamma(o, \tilde{x}_n, \lfloor d(o, \tilde{x}_n) \rfloor))
    \leq \lambda(\theta(0) + 2) + 3k + 1.
    \]
    By the proof of \cref{lem:qg-const-geodesic}, we have
    $|R - d(o, \tilde{x}_n)| \leq \theta(0)$.
    Therefore, for any $i \in \mathbb{N}$ with $i \geq N_m$, there is some $n, l \in \mathbb{N}$ such that $n, l \geq i$ and we have 
    \begin{align*}
        &d(\texttt{rp}\Gamma(o, y_{m,l}, R), \texttt{rp}\Gamma(o, x_n, R))\\
        &\leq d(\texttt{rp}\Gamma(o, y_{m,l}, R), \texttt{rp}\Gamma(o, y_{m,l}, d(o, \tilde{x}_n)))
        + d(\texttt{rp}\Gamma(o, y_{m,l}, d(o, \tilde{x}_n)), \texttt{rp}\Gamma(o, y_{m,l}, \lfloor d(o, \tilde{x}_n)\rfloor))\\
        &\phantom{ \leq \quad} 
        + d(\texttt{rp}\Gamma(o, y_{m,l}, \lfloor d(o, \tilde{x}_n)\rfloor), \texttt{rp}\Gamma(o, \tilde{x}_n, \lfloor d(o, \tilde{x}_n) \rfloor))\\
        &\phantom{\leq \quad}
        + d(\texttt{rp}\Gamma(o, \tilde{x}_n, \lfloor d(o, \tilde{x}_n) \rfloor), \texttt{rp}\Gamma(o, x_n, \lfloor d(o, \tilde{x}_n) \rfloor)\\
        &\phantom{\leq \quad}
        + d(\texttt{rp}\Gamma(o, x_n, \lfloor d(o, \tilde{x}_n) \rfloor), \texttt{rp}\Gamma(o, x_n, d(o, \tilde{x}_n)))
        + d(\texttt{rp}\Gamma(o, x_n, d(o, \tilde{x}_n)), \texttt{rp}\Gamma(o, x_n, R))\\
        &\leq \lambda(4\theta(0) + 6) + 10k + 1.
    \end{align*}
    Set $L \coloneqq \lambda(4\theta(0) + 6) + 10k + 1$.
    Note $L > 1$.
    For any $R > 0$, there is some $M \in \mathbb{N}$ such that for any $m \in \mathbb{N}$ with $m \geq M$, there is some $N_m \in \mathbb{N}$ such that for any $i \in \mathbb{N}$ with $i \geq N_m$, there are some $n, l \in \mathbb{N}$ such that $n, l \geq i$
    and we have
    \[
    d(\texttt{rp}\Gamma(o, y_{m,l}, R), \texttt{rp}\Gamma(o, x_n, R)) \leq L.
    \]

    We show, for any $s \in \mathbb{R}_{\geq 0}$, there is some $M' \in \mathbb{N}$ such that for any $m \in \mathbb{N}$ with $m \geq M'$, we have $\liminf_{n,l \rightarrow \infty}(x_n \:|\: y_{m, l})_o \geq s$.
    From above, there is some $M' \in \mathbb{N}$ such that for any 
    $m \in \mathbb{N}$ with $m \geq M'$, 
    there is some $N_m \in \mathbb{N}$ such that for any $i \in \mathbb{N}$ with $i \geq N_m$, there are some $n, l \in \mathbb{N}$ such that $n, l \geq i$ and we have
    \[
    d(\texttt{rp}\Gamma(o, y_{m,l}, ELs), \texttt{rp}\Gamma(o, x_n, ELs)) \leq L.
    \]
    Since $\Gamma$ is a $(\lambda, k, E, C, \theta)$-coarsely convex bicombing, we have
    \begin{align*}
         d(\texttt{rp}\Gamma(o, y_{m,l}, s), \texttt{rp}\Gamma(o, x_n, s))
        &= d\left(\texttt{rp}\Gamma\left(o, y_{m,l},\frac{1}{EL} EL s\right), \texttt{rp}\Gamma\left(o, x_n, \frac{1}{EL} ELs\right)\right)\\
        &\leq \frac{1}{EL}Ed(\texttt{rp}\Gamma(o, y_{m,l}, ELs), \texttt{rp}\Gamma(o, x_n, ELs)) + C\\
        &\leq 1 + C \leq D_1.
    \end{align*}
    This means that for any $s \in \mathbb{R}_{\geq 0}$, there is some $M' \in \mathbb{N}$ such that for any $m \in \mathbb{N}$ with $m \geq M'$, we have $\liminf_{n,l \rightarrow \infty}(x_n \:|\: y_{m, l})_o \geq s$.
    Therefore, we have 
    \[
    \lim_{m \rightarrow \infty}\liminf_{n, l \rightarrow \infty}(x_n \:|\: y_{m, l})_o = \infty.
    \]
\end{proof}

\subsection{Reduced horoboundary}

\begin{definition}\label{def:reduced-horoboundary}
    Let $(X,d)$ be a proper $(\lambda, k, E, C, \theta)$-coarsely convex space and let $\partial^c_h X$ be the horoboundary defined above, using the cone metric.

 We define that two functions $\xi$ and $\eta$ in $\partial^c_h X$ are equivalent,
 denoted by $\xi \sim \eta$, if 
    \[
    \sup_{u \in X}|\xi(u) - \eta(u)| < \infty.
    \]

 This determines an equivalence relation on 
 $\partial^c_h X$. 
 A \emph{reduced horoboundary of $X$ with cone metric}, 
 denoted by $\partial^c_h X/\sim$, is the quotient of 
 $\partial^c_h X$ by the equivalence relation $\sim$ with the quotient topology.
 We also say that $\partial^c_h X/\sim$ is a reduced horoboundary of $(X,d_c)$.
\end{definition}

Let $\pi \colon \partial^c_h X \rightarrow \partial^c_h X/\sim$ be the natural projection.
If $\xi \sim \eta$, 
we have $\mathrm{Pr}(\xi) = \mathrm{Pr}(\eta)$.
So there uniquely exists a continuous map 
$f \colon \partial^c_h X/\sim \, \rightarrow \partial_o X$ 
such that $\mathrm{Pr} = f \circ \pi$.
We construct the inverse of $f$.

\begin{definition}
    For any $x \in \partial_o X$, we define a $(\lambda, k_1)$-quasi geodesic ray 
    $\gamma_x \colon \mathbb{R}_{\geq 0} \rightarrow X$ as follows
    \[
      \gamma_x(-) \coloneqq \texttt{rp}\bar{\Gamma}(o, x, -)
    \]
    where $k_1 \coloneqq \lambda + k$.
    Then we define $b_{x} \colon X \rightarrow \mathbb{R}$
    as the \emph{Busemann function respect to the cone metric} by
    \[
    u \mapsto -d(o, u) + d(\texttt{rp}\Gamma(o, u, \lfloor d(o, u) \rfloor), \gamma_x(\lfloor d(o, u) \rfloor)).
    \]
\end{definition}

\begin{proposition}\label{prop:qg-g-is-well-defined}
    The Busemann function respect to the cone metric is a horofunction.
\end{proposition}

\begin{proof}
Let $x \in \partial_o X$ and define
$\gamma_x \colon \mathbb{R}_{\geq 0} \rightarrow X$ by
$\gamma_x(-) = \texttt{rp}\bar{\Gamma}(o, x, -)$.
Then, there is a sequence $\{x_n\}_{n \in \mathbb{N}} \in S^{\infty}_o X$ such that 
we have $(x_n \:|\: x)_o \rightarrow \infty$ as $n \rightarrow \infty$ and 
the sequence of maps
$\{\texttt{rp}\Gamma(o, x_n, -)|_{\mathbb{N}} \colon \mathbb{N} \rightarrow X\}_{n \in \mathbb{N}}$ converges pointwise to
$\gamma_x(-) = \texttt{rp}\bar{\Gamma}(o, x, -)$.

First, we show $b_x \in \cl \psi(X)$.
For any $R > 0$, by taking $n \in \mathbb{N}$ large enough, we have
\begin{align*}
    &\sup_{u \in \bar{B}(o, R)}|\psi_{x_n}(u) - b_x(u)|\\
    &= \sup_{u \in \bar{B}(o, R)}|\{d(o, x_n) - d(o, u) + d(\texttt{rp}\Gamma(o, x_n, \lfloor d(o, u) \rfloor), \texttt{rp}\Gamma(o, u, \lfloor d(o, u) \rfloor)) - d(o, x_n)\}\\
    &\phantom{\sup_{u \in \bar{B}(o, R)} \quad \quad}
    - \{- d(o, u) + d(\texttt{rp}\Gamma(o, u, \lfloor d(o, u) \rfloor), \gamma_x(\lfloor d(o, u) \rfloor))\}\\
    &= \sup_{u \in \bar{B}(o, R)}|d(\texttt{rp}\Gamma(o, x_n, \lfloor d(o, u) \rfloor), \texttt{rp}\Gamma(o, u, \lfloor d(o, u) \rfloor) - d(\texttt{rp}\Gamma(o, u, \lfloor d(o, u) \rfloor, \gamma_x(\lfloor d(o, u) \rfloor))|\\
    &\leq \sup_{u \in \bar{B}(o, R)}d(\texttt{rp}\Gamma(o, x_n, \lfloor d(o, u) \rfloor), \gamma_x(\lfloor d(o, u) \rfloor))\\
    &= \max_{t \in \{1, 2, \cdots, \lfloor  R \rfloor\}}d(\texttt{rp}\Gamma(o, x_n, t), \gamma_x(t)) \rightarrow 0 \text{ as } n \rightarrow \infty.
\end{align*}

Second, we show $b_x \notin B_{b, \lambda, E}(X)$ by contradiction. That is, there are some $y \in X$
and $\beta \in B_b(X)$ such that we have

\begin{align}
 \label{eq:pshi_y-b_x_is_bounded}
 \frac{1}{\lambda E}\psi_y - b_x \leq \beta. 
\end{align}

We take $R > 0$ with
\[
d(o, y) \leq \frac{1}{\lambda}R - k.
\]
Set $\tilde{x}_n \coloneqq \texttt{rp}\Gamma(o, x_n, R)$.
Then, for any $n \in \mathbb{N}$, we have
\[
d(o, y) \leq \frac{1}{\lambda}R - k \leq d(o, \tilde{x}_n) \leq \lambda R + k.
\]
For large enough $n \in \mathbb{N}$, we have $d(o, x_n) > \lambda R + k$.

By the proof of \cref{lem:qg-const-geodesic}, we have
\[
d(\texttt{rp}\Gamma(o, x_n, \lfloor d(o, \tilde{x}_n) \rfloor), \texttt{rp}\Gamma(o, \tilde{x}_n, \lfloor d(o, \tilde{x}_n) \rfloor)) \leq \lambda(\theta(0) + 2) + 3k.
\]
Thus, we have
\begin{align*}
    &\frac{1}{\lambda E}\psi_y(\tilde{x}_n) - \psi_{x_n}(\tilde{x}_n)\\
    &= \frac{1}{\lambda E}\{d(o, \tilde{x}_n) - d(o, y) + d(\texttt{rp}\Gamma(o, \tilde{x}_n,\lfloor d(o, y) \rfloor), \texttt{rp}\Gamma(o, y, \lfloor d(o, y) \rfloor)) - d(o, y)\}\\
    &\phantom{ {=} \quad}
    -\{d(o, x_n) - d(o, \tilde{x}_n) + d(\texttt{rp}\Gamma(o, x_n, \lfloor d(o, \tilde{x}_n) \rfloor), \texttt{rp}\Gamma(o, \tilde{x}_n, \lfloor d(o, \tilde{x}_n)\rfloor)) - d(o, x_n)\}\\
    &= \left(\frac{1}{\lambda E} + 1 \right)d(o, \tilde{x}_n) - \frac{2}{\lambda E}d(o, y) + \frac{1}{\lambda E}d(\texttt{rp}\Gamma(o, \tilde{x}_n,\lfloor d(o, y) \rfloor), \texttt{rp}\Gamma(o, y, \lfloor d(o, y) \rfloor))\\
    &\phantom{ {=} \quad}
    - d(\texttt{rp}\Gamma(o, x_n, \lfloor d(o, \tilde{x}_n) \rfloor), \texttt{rp}\Gamma(o, \tilde{x}_n, \lfloor d(o, \tilde{x}_n)\rfloor))\\
    &\geq \left(\frac{1}{\lambda E} + 1 \right) \left(\frac{1}{\lambda}R - k \right) 
    - \frac{2}{\lambda E}d(o, y)
    - \lambda(\theta(0) + 2) - 3k.
\end{align*}
Therefore, we have
\begin{align*}
    \sup_{u \in \bar{B}(o, \lambda R + k)}
    \frac{1}{\lambda E}\psi_y(u) - \psi_{x_n}(u) 
    \geq \left(\frac{1}{\lambda E} + 1 \right) \left(\frac{1}{\lambda}R - k \right) 
    - \frac{2}{\lambda E}d(o, y)
    - \lambda(\theta(0) + 2) - 3k
\end{align*}
for large enough $n \in \mathbb{N}$.

On the other hand, since $\psi_{x_n}$ converges uniformly to $b_x$ on compact sets, we have 
\[
\lim_{n \rightarrow \infty}
\sup_{u \in \bar{B}(o, \lambda R + k)}
\frac{1}{\lambda E}\psi_y(u) - \psi_{x_n}(u) 
= \sup_{u \in \bar{B}(o, \lambda R + k)}
    \frac{1}{\lambda E}\psi_y(u) - b_x(u).
\]
So we have
\[
\sup_{u \in \bar{B}(o, \lambda R + k)}
    \frac{1}{\lambda E}\psi_y(u) - b_x(u)
    \geq \left(\frac{1}{\lambda E} + 1 \right) \left(\frac{1}{\lambda}R - k \right) 
    - \frac{2}{\lambda E}d(o, y)
    - \lambda(\theta(0) + 2) - 3k.
\]
Since the right-hand side of the above equation contains $R > 0$ 
which can be arbitrary large, this contradicts to
\cref{eq:pshi_y-b_x_is_bounded}.
\end{proof}

For $x \in \partial_o X$, we define $g \colon \partial_o X \rightarrow \partial^c_h X / \sim$ as $g(x) \coloneqq \pi(b_x)$.

\begin{proposition}
\label{prop:g-is-inverse-of-f}
    The map $g$ is the inverse of $f$.
\end{proposition}

\begin{proof}
    First, we show $f \circ g = \mathrm{id}_{\partial_o X}$.
    We take $x \in \partial_o X$ arbitrarily. 
    Then, we have $g(x) = \pi(b_x)$.
    By \cref{prop:qg-g-is-well-defined}, there is a sequence 
    $\{x_n\}_{n \in \mathbb{N}} \in S^{\infty}_o X$ such that 
    we have
    $(x_n \:|\: x)_o \rightarrow \infty$
     as  
     $n \rightarrow \infty$
     and 
    $\lim_{n \rightarrow \infty}\psi_{x_n} = b_x$.

    Thus, we obtain
    \[
    f \circ g(x) = f(\pi(b_x)) = \mathrm{Pr}(b_x) = x.
    \]

    Second, we show $g \circ f = \mathrm{id}_{\partial^c_h X / \sim}$.
    We take $\xi \in \partial^c_h X$ arbitrarily.
    By \cref{lem:qg-const-geodesic}, there is a sequence of maps $\{\psi_{y_n}\}_{n \in \mathbb{N}}$ converges uniformly to $\xi$ on compact sets, and we have $\{y_n\}_{n \in \mathbb{N}} \in S^{\infty}_o X$.
    Let $y \in \partial_o X$ be an equivalence class of a sequence $\{y_n\}_{n \in \mathbb{N}} \in S^{\infty}_o X$.
    By the definition of $\mathrm{Pr}$, we have $\mathrm{Pr}(\xi) = y$.

    There is a sequence $\{z_n\}_{n \in \mathbb{N}} \in S^{\infty}_o X$ such that we have 
    $(z_n \:|\: y)_o \rightarrow  \infty$ as $n \rightarrow \infty$
    and the sequence of maps $\{\psi_{z_n}\}_{n \in \mathbb{N}}$ converges uniformly to $b_y$ on compact sets.
    By the definition of $f$ and $g$, we have
    \[
    g \circ f(\pi(\xi)) = g(\mathrm{Pr}(\xi)) = g(y) = \pi(b_y).
    \]
    Thus, it is enough to show
    \[
    \sup_{u \in X}|\xi(u) - b_y(u)| 
    = \sup_{u \in X}\lim_{n \rightarrow \infty}|\psi_{y_n}(u) - \psi_{z_n}(u)| < \infty.
    \]
    
    We take $u \in X$ arbitrarily.
    For large enough $n \in \mathbb{N}$, we have
    \begin{align*}
       &|\psi_{y_n}(u) - \psi_{z_n}(u)|\\
       &= |\{d(o, y_n) - d(o, u) + d(\texttt{rp}\Gamma(o, y_n, \lfloor d(o, u) \rfloor), \texttt{rp}\Gamma(o, u, 
       \lfloor d(o, u) \rfloor)) - d(o, y_n)\}\\
       & \phantom{ = \quad}
       -\{d(o, z_n) - d(o, u) + d(\texttt{rp}\Gamma(o, z_n, \lfloor d(o, u) \rfloor), \texttt{rp}\Gamma(o, u, 
       \lfloor d(o, u) \rfloor)) - d(o, z_n)\}|\\
       &\leq d(\texttt{rp}\Gamma(o, y_n, \lfloor d(o, u) \rfloor), \texttt{rp}\Gamma(o, z_n, \lfloor d(o, u) \rfloor))
    \end{align*}

    By $(y_n \:|\: y)_o \rightarrow \infty$ and $(z_n \:|\: y)_o \rightarrow \infty$ as $n \rightarrow \infty$, we have
    \[
    \liminf_{n \rightarrow \infty}(y_n \:|\: z_n)_o = \infty.
    \]

    Thus, there is some $N \in \mathbb{N}$ such that for any $n \in \mathbb{N}$ with $n \geq N$, there is some $n' \in \mathbb{N}$ with $n' \geq n$ such that we have 
    \[
    \sup\{ t \in \mathbb{R}_{\geq 0} \:|\: d(\texttt{rp}\Gamma(o, y_{n'}, t), \texttt{rp}\Gamma(o, z_{n'}, t)) \leq D_1\}
    \geq \lfloor d(o, u) \rfloor.
    \]
    So, there is some $s \geq \lfloor d(o, u) \rfloor$ such that we have 
    \[
    d(\texttt{rp}\Gamma(o, y_{n'}, s), \texttt{rp}\Gamma(o, z_{n'}, s)) \leq D_1.
    \]
    Therefore, there is some $N \in \mathbb{N}$ such that for any $n \in \mathbb{N}$ with $n \geq N$, there is some 
    $n' \in \mathbb{N}$ with $n' \geq n$ such that we have
    \begin{align*}
         &d(\texttt{rp}\Gamma(o, y_{n'}, \lfloor d(o, u) \rfloor), \texttt{rp}\Gamma(o, z_{n'}, \lfloor d(o, u) \rfloor)\\
         &=  d\left(\texttt{rp}\Gamma \left(o, y_{n'}, \frac{\lfloor d(o, u) \rfloor}{s}s \right), \texttt{rp}\Gamma \left(o, z_{n'}, \frac{\lfloor d(o, u) \rfloor}{s}s \right)\right)\\
         &\leq \frac{\lfloor d(o, u) \rfloor}{s}Ed(\texttt{rp}\Gamma(o, y_{n'}, s), \texttt{rp}\Gamma(o, z_{n'}, s)) + C\\
         &\leq ED_1 + C.
    \end{align*}
    From the above, we obtain
    \begin{align*}
        |\xi(u) - b_y(u)| &= \lim_{n \rightarrow \infty}|\psi_{y_n}(u) - \psi_{z_n}(u)|\\
        &= \liminf_{n \rightarrow \infty}|\psi_{y_n}(u) - \psi_{z_n}(u)|\\
        &\leq \liminf_{n \rightarrow \infty}d(\texttt{rp}\Gamma(o, y_n, \lfloor d(o, u) \rfloor), \texttt{rp}\Gamma(o, z_n, \lfloor d(o, u) \rfloor))\\
        &\leq ED_1 + C
    \end{align*}
    for any $u \in X$.
    So, we have
    \[
    \sup_{u \in X}|\xi(u) - b_y(u)| \leq ED_1 + C.
    \]
\end{proof}

We show that $g \colon \partial_o X \rightarrow \partial^c_h X/\sim$ is continuous.
We recall the following lemma
from general topology.

\begin{lemma}\label{lem:general-topology}
 Let $(Z, \mathcal{O})$ be a topological space and 
 let $\alpha \in Z$ be a point on $Z$. 
 Let $\{a_n\}_{n \in \mathbb{N}}$ be a sequence on $Z$ 
 such that for any subsequence $\{a_{n(m)}\}_{m \in \mathbb{N}}$ 
 has a subsequence $\{a_{n(m)(k)}\}_{k \in \mathbb{N}}$ such that 
 $a_{n(m)(k)}$ converges to $\alpha$. Then $a_n$ converges to $\alpha$.
\end{lemma}

\begin{proof}
    We prove this by contradiction. 
    That is, there is an open set $U \in \mathcal{O}$ such that $\alpha \in U$ and for any $N \in \mathbb{N}$, 
    there is some $n \in \mathbb{N}$, $n \geq N$ and $a_n \notin U$. 
    We construct a subsequence
    $\{a_{n(m)}\}_{m \in \mathbb{N}}$ of $\{a_n\}_{n \in \mathbb{N}}$ 
    by induction as follows.
    
    By the assumption, 
 there is $n(1) \in \mathbb{N}$ such that $n(1) \geq 1$ and $a_{n(1)} \notin U$. 
    Also, there is $n(2) \in \mathbb{N}$ such that $n(2) > n(1)$ and $a_{n(2)} \notin U$.
    In the same way, for any $m \in \mathbb{N}$, 
    there is $n(m+1) \in \mathbb{N}$ such that $n(m+1) > n(m)$ and $a_{n(m+1)} \notin U$. 
 Thus, we obtain
 a subsequence $\{a_{n(m)}\}_{m \in \mathbb{N}}$ 
 that does not intersect with $U$. 
    By the construction of 
    $\{a_{n(m)}\}_{m \in \mathbb{N}}$, 
    there is no subsequence of $\{a_{n(m)}\}_{m \in \mathbb{N}}$ that converges to $\alpha$. This is a contradiction.
\end{proof}

\begin{proposition}\label{prop:g-is-continuous}
    $g \colon \partial_o X \rightarrow \partial^c_h X / \sim$
    is continuous.
\end{proposition}

\begin{proof}
    We take $x \in \partial_o X$ arbitrarily, and we take a sequence 
    $\{x_n\}_{n \in \mathbb{N}} \subset \partial_o X$ 
    arbitrarily such that 
    $\liminf_{n \rightarrow \infty}(x_n \:|\: x)_o = \infty$. 
    By \cref{prop:qg-g-is-well-defined}, we have
    $b_{x} \in \partial^c_h X$ and $b_{x_n} \in \partial^c_h X$ 
    for any $n \in \mathbb{N}$.
    Then we show $\lim_{n \rightarrow \infty}\pi(b_{x_n}) = \pi(b_x)$.

    Set $\gamma_x(-) \coloneqq \text{rp}\bar{\Gamma}(o, x, -)$ and 
    $\gamma_{x_n}(-) \coloneqq \text{rp}\bar{\Gamma}(o, x_n, -)$ for any $n \in \mathbb{N}$.
     By \cref{lem:gromov-product}, we have 
    $\liminf_{n \rightarrow \infty}(\gamma_{x_n} \:|\: \gamma_x)_o = \infty$.
    
    We take any subsequence 
    $\{\pi(b_{x_{n(m)}})\}_{m \in \mathbb{N}}$ of 
    $\{\pi(b_{x_n})\}_{n \in \mathbb{N}}$.
    Correspondingly, we  take a sequence of $(\lambda, k_1)$-quasi geodesic rays
    $\{\gamma_{x_{n(m)}}\}_{m \in \mathbb{N}}$ where 
    $k_1 = \lambda + k$.
    By the properness of $X$ and induction, 
    we can take a subsequence 
    $\{\gamma_{x_{n(m)(l)}}\}_{l \in \mathbb{N}}$ 
    of $\{\gamma_{x_{n(m)}}\}_{m \in \mathbb{N}}$ 
    such that $\{\gamma_{x_{n(m)(l)}}\}_{l \in \mathbb{N}}$ 
    converges uniformly to some $(\lambda, k_1)$-quasi geodesic ray $\tilde{\gamma}$ on compact sets.
    For details, see \cite[Proposition 4.17]{FO17}.
    We define $b_{\tilde{\gamma}} \colon X \rightarrow \mathbb{R}$ as follows
    \[
    u \mapsto -d(o, u) + d(\texttt{rp}\Gamma(o, u, \lfloor d(o, u) \rfloor), \tilde{\gamma}(\lfloor d(o, u) \rfloor))
    \]
    for any $u \in X$.
    Then, 
    for any $R > 0$, we have
    \begin{align*}
        &\sup_{u \in \bar{B}(o, R)}|b_{x_{n(m)(l)}}(u) - b_{\tilde{\gamma}}(u)|\\
        &= \sup_{u \in \bar{B}(o, R)}|\{-d(o, u) + d(\texttt{rp}\Gamma(o, u, \lfloor d(o, u) \rfloor), \gamma_{x_{n(m)(l)}}(\lfloor d(o, u) \rfloor))\}\\
        &\phantom{ = \quad}- \{-d(o, u) + d(\texttt{rp}\Gamma(o, u, \lfloor d(o, u) \rfloor), \tilde{\gamma}(\lfloor d(o, u)\rfloor)\}|\\
        &\leq \sup_{u \in \bar{B}(o, R)}d(\gamma_{x_{n(m)(l)}}(\lfloor d(o, u) \rfloor), \tilde{\gamma}(\lfloor d(o, u) \rfloor))\\
        &\leq \max_{t \in \{1, \cdots ,\lfloor R \rfloor\}}d(\gamma_{x_{n(m)(l)}}(t), \tilde{\gamma}(t)) \rightarrow 0 \text{ as } l \rightarrow \infty.
    \end{align*}
    Thus, $\{b_{x_{n(m)(l)}}\}_{l \in \mathbb{N}}$ converges uniformly to $b_{\tilde{\gamma}}$ on compact sets.
    Since $b_{x_{n(m)(l)}} \in \partial^c_h X \subset \cl\psi(X)$ for any $l \in \mathbb{N}$, we have 
    $\lim_{l \rightarrow \infty}b_{x_{n(m)(l)}} = b_{\tilde{\gamma}} \in \cl\psi(X)$.
    Also, we have $b_{\tilde{\gamma}} \notin B_{b, \lambda, E}$.
    We take $R > 0$ and $y \in X$ arbitrarily. 
    By the proof of \cref{prop:qg-g-is-well-defined}, we have 
    \[
    \sup_{u \in \bar{B}(o, \lambda R+ k)}\frac{1}{\lambda E}\psi_y(u) - b_{x_{n(m)(l)}}(u) \geq \left(\frac{1}{\lambda E} + 1 \right)\left(\frac{1}{\lambda}R - k \right) - \frac{2}{\lambda E}d(o, y) - \lambda(\theta(0) + 2) - 3k
    \]
    for any $l \in \mathbb{N}$.
    Since $b_{x_{n(m)(l)}}$ converges uniformly to $b_{\tilde{\gamma}}$ on compact sets, we have
    \[
    \sup_{u \in \bar{B}(o, \lambda R+ k)}\frac{1}{\lambda E}\psi_y(u) - b_{\tilde{\gamma}}(u) \geq \left(\frac{1}{\lambda E} + 1 \right)\left(\frac{1}{\lambda}R - k \right) - \frac{2}{\lambda E}d(o, y) - \lambda(\theta(0) + 2) - 3k
    \]
    for any $R > 0$ and $y \in X$.
    This means $b_{\tilde{\gamma}} \notin B_{b, \lambda, E}(X)$
    and we have $b_{\tilde{\gamma}} \in \partial^c_h X$.
    
    We show $\pi(b_{\tilde{\gamma}}) = \pi(b_x)$, 
    that is $\sup_{u \in X}|b_{\tilde{\gamma}}(u) - b_x(u)| < \infty$.
    We prove this by contradiction.
    We assume that 
    \[
    \sup_{u \in X}|b_{\tilde{\gamma}}(u) - b_x(u)| = \sup_{u \in X}\lim_{l \rightarrow \infty}|b_{x_{n(m)(l)}}(x) - b_x(u)| = \infty.
    \]
    For any $M > D_1 + C$, there is $x_M \in X$ such that for any $L \in \mathbb{N}$, there is some $l \in \mathbb{N}$, $l \geq L$ such that
    \[
    |b_{\tilde{\gamma}}(x_M) - b_{x_{n(m)(l)}}(x_M)| 
    = |d(x_M, \tilde{\gamma}(d(o,x_M))) - d(x_M, \gamma_{x_{n(m)(l)}}(d(o, x_M)))| 
    > M.
    \]
    By the triangle inequality, for any $L \in \mathbb{N}$ there is some $l \in \mathbb{N}$ such that 
    $l \geq L$ and we have
    \[
    d(\tilde{\gamma}(d(o, x_M)), \gamma_{x_{n(m)(l)}}(d(o, x_M))) > M > D_1 + C.
    \]
    On the other hand, since $\{\gamma_{x_{n(m)(l)}}\}_{l \in \mathbb{N}}$ is a subsequence of 
    $\{\gamma_{x_n}\}_{n \in \mathbb{N}}$, we have
    \[
    \liminf_{l \rightarrow \infty}(\gamma_{x_{n(m)(l)}} \:|\: \gamma_x)_o = \infty.
    \]
    Thus, there is some $\tilde{L} \in \mathbb{N}$ such that for any $\tilde{l} \in \mathbb{N}$, $\tilde{l} \geq \tilde{L}$, there is $t_M > d(o, x_M)$ such that
    $d(\gamma_{x_{n(m)(\tilde{l})}}(t_M), \gamma_x(t_M)) \leq D_1$. Thus, we have
    \begin{align*}
        d(\gamma_{x_{n(m)(\tilde{l})}}(d(o, x_M)), \gamma_x(d(o, x_M)) 
        &= d(\texttt{rp}\bar{\Gamma}(o, x_{n(m)(\tilde{l})}, d(o, x_M)), \texttt{rp}\bar{\Gamma}(o, x, d(o, x_M)))\\
        &\leq \frac{d(o, x_M)}{t_M}d(\texttt{rp}\bar{\Gamma}(o, x_{n(m)(\tilde{l})}, t_M), \texttt{rp}\bar{\Gamma}(o, x, t_M)) + C \\
        &= \frac{d(o, x_M)}{t_M}d(\gamma_{x_{n(m)(\tilde{l})}}(t_M), \gamma_x(t_M)) + C\\
        &\leq D_1 + C < M.
    \end{align*}
    This is a contradiction. So we have
    \[
    \sup_{u \in X}|b_{\tilde{\gamma}}(u) - b_x(u)| 
    <\infty.
    \]
    In particular, we have
    \[
    \lim_{l \rightarrow \infty}\pi(b_{x_{n(m)(l)}}) 
    = \pi(\lim_{l \rightarrow \infty}b_{x_{n(m)(l)}})
    = \pi(b_{\tilde{\gamma}})
    = \pi(b_x).
    \]
    From above, any subsequence $\{\pi(b_{x_{n(m)}})\}_{m \in \mathbb{N}}$ of $\{\pi(b_{x_n})\}_{n \in \mathbb{N}}$ has a subsequence which converges to $\pi(b_x)$. By \cref{lem:general-topology}, we have 
    $\lim_{n \rightarrow \infty}\pi(b_{x_n}) = \pi(b_x)$.
\end{proof}

Combining \cref{prop:qg-Pr-is-continuous,prop:g-is-inverse-of-f,prop:g-is-continuous}, we complete the proof of \cref{thm-Intro:reduced-horobdry-equal-ideal-bdry}.


\section*{Acknowledgement}
I would like to thank Professor Tomohiro Fukaya for giving me tremendous opportunities for growth through weekly instruction and off-campus placements over the past four years.

I would like to thank Dr. Kenshiro Tashiro for providing us with very useful information on horoboundary and for his very careful guidance.

I would also like to thank Dr. Yoshito Ishiki, Dr. Yuya Kodama, 
and Dr. Takumi Matsuka for their constant encouragement.

Finally, I would like to thank my parents for allowing me to study for a total of six years at the university and graduate school.

\bibliographystyle{alpha}
\bibliography{cc-horobundary}
\end{document}